\documentclass[12pt]{amsart}
\usepackage{amssymb}
\usepackage{amsmath}
\usepackage{amsthm} 
\usepackage{graphicx,subcaption}
\usepackage[utf8]{inputenc}
\usepackage{amsfonts}
\usepackage{graphicx}
\usepackage[hidelinks]{hyperref}
\usepackage{xcolor}
\numberwithin{equation}{section}
\newtheorem{theorem}{Theorem}[section]

\newtheorem{lemma}[theorem]{Lemma}
\newtheorem{claim}[theorem]{Claim}

\newtheorem{example}[theorem]{Example}
 \newtheorem{remark}[theorem]{Remark}\newtheorem{proposition}[theorem]{Proposition}
\newtheorem{corollary}[theorem]{Corollary} \newcommand{\La}{\Lambda}
\newcommand{\R}{{\mathbb R}} \newcommand{\Q}{{\mathbb Q}} \newcommand{\Z}{{\mathbb Z}} \newcommand{\N}{{\mathbb N}}

\newcommand{\Cc}{{\mathbb C}}\newcommand{\vv}{{\bf v}}   
\newcommand{\supp}{{\rm supp\,}}

\title[Sampling in the SIS generated by the bivariate Gaussian]{Sampling in the shift-invariant space generated by the bivariate Gaussian function}
\thanks{This research was funded in whole or in part by the Austrian Science Fund (FWF): 10.55776/Y1199 (J. L. R.) and 10.55776/P33217 (I. Z.). For open access purposes, the authors have applied a CC BY public copyright license to any author-accepted manuscript version arising from this submission.}

\author{Jos\'e Luis Romero}
\address[J. L. Romero]{Faculty of Mathematics, University of Vienna, Oskar-Morgenstern-Platz 1, A-1090 Vienna, Austria, and Acoustics Research Institute, Austrian Academy of Sciences, Dr. Ignaz Seipel-Platz 2,	AT-1010 Vienna, Austria}
\email{jose.luis.romero@univie.ac.at}

 \author{Alexander Ulanovskii}
 \address[A. Ulanovskii]{Department of Mathematics and Physics, University of Stavanger, 4036 Stavanger, Norway}
 \email{alexander.ulanovskii@uis.no}
 
 \author{Ilya Zlotnikov}
 \address[I. Zlotnikov]{Faculty of Mathematics, University of Vienna, Oskar-Morgenstern-Platz 1,
		A-1090 Vienna, Austria}
\email{ilia.zlotnikov@univie.ac.at}



\begin{document}
\begin{abstract}
We study the space spanned by the integer shifts of a bivariate Gaussian function and the problem of reconstructing any function in that space from samples scattered across the plane. We identify a large class of lattices, or more generally semi-regular sampling patterns spread along parallel lines, that lead to stable reconstruction while having densities close to the critical value given by Landau's limit. At the critical density, we construct examples of sampling patterns for which reconstruction fails.

In the same vein, we also investigate continuous sampling along non-uniformly scattered families of parallel lines and identify the threshold density of line configurations at which reconstruction is possible. In a remarkable contrast with Paley-Wiener spaces, the results are completely different for lines with rational or irrational slopes.

Finally, we apply the sampling results to Gabor systems with bivariate Gaussian windows. As a main contribution, we provide a large list of new examples of Gabor frames with non-complex lattices having volume close to 1.
\end{abstract}
 \date{}\maketitle

\section{Introduction}
\subsection{Sampling sets}
In this article we study sampling in the shift-invariant space generated by the bivariate Gaussian function: 
\begin{align}\label{eq_i2}
V^2_a(\mathbb{R}^2):=\Big\{f(x,y)=\sum_{(n,m)\in\mathbb Z^2}c_{n,m}e^{-a(x-n)^2-a(y-m)^2}:c\in \ell^2(\mathbb Z^2)\Big\}.
\end{align}
Here $a>0$ is a \emph{shape parameter}, and our results are mostly independent of it. The main problem is to determine if a given discrete set $\Lambda \subset \mathbb{R}^2$ is a \emph{sampling set {\rm(}SS{\rm)}} for the space $V^2_a(\mathbb{R}^2)$, that is, whether there exist positive constants $A,B$ such that
\begin{align}\label{eq_i1}
A\|f\|_2^2 \leq \sum_{\lambda \in \Lambda} |f(\lambda)|^2 \leq B \|f\|_2^2, \qquad f \in V^2_a(\mathbb{R}^2).
\end{align}
The basic intuition is that functions in the space $V^2_a(\mathbb{R}^2)$ have one degree of freedom per unit area. One may thus expect a set $\Lambda \subset \mathbb{R}^2$ to satisfy the the \emph{sampling inequalities}
\eqref{eq_i1} if it has, on average, at least one point per unit area.

Conclusive results consistent with that intuition are available for the shift-invariant space generated by the univariate Gaussian function:
\begin{align}\label{eq_i2p}
V^2_a(\mathbb{R}):=\Big\{f(x)=\sum_{n\in\mathbb Z}c_ne^{-a(x-n)^2}:c\in \ell^2(\mathbb Z)\Big\}.
\end{align}
The results concern a \emph{separated set} $\Gamma \subseteq \mathbb{R}$, i.e.,
\begin{equation}\label{sep_const}
\delta(\Gamma):= \inf_{\gamma,\gamma'\in\Gamma, \gamma\not=\gamma'} |\gamma-\gamma'| > 0    
\end{equation}
and are formulated in terms of Beurling's lower density
\begin{align}\label{eq_i5}
	D^-(\Gamma) = \liminf_{R\to\infty} \inf_{x\in \mathbb{R}} \tfrac{1}{R} \# \Gamma \cap [x-R/2,x+R/2].
\end{align}
The following is a special case of \cite[Theorem 4.4]{grs}. 
\begin{theorem}\label{th_i3}
Let $\Gamma \subseteq \mathbb{R}$ be separated.
\begin{itemize}
\item[(a)] If $D^-(\Gamma)>1$ then $\Gamma$ is a sampling set for $V^2_a(\mathbb{R})$.
\item[(b)] If $D^-(\Gamma)<1$ then $\Gamma$ is not a sampling set for $V^2_a(\mathbb{R})$.
\end{itemize}
In addition, at the critical density $D^-(\Gamma)=1$ both situations are possible: for example, $\Gamma=\mathbb{Z}$ is a sampling set for $V^2_a(\mathbb{R})$ while $\Gamma=\mathbb{Z}\setminus\{0\}$ is not. 
\end{theorem}
Theorem \ref{th_i3} parallels a classical result for
\emph{bandlimited functions}, that is, functions whose
Fourier transforms $\hat{f}(\xi) = \int_\R f(x) e^{-2\pi i x \xi} \,dx$ are compactly supported. Indeed, the statement of Theorem \ref{th_i3} holds verbatim 
if $V^2_a(\mathbb{R})$ is replaced by the \emph{Paley-Wiener space}:
\begin{align*}
PW_{[-1/2,1/2]}(\R) = \big\{f \in L^2(\R):\, \supp(\hat{f}) \subset [-1/2,1/2]\big\}.
\end{align*}

The sampling theory for bivariate functions is comparatively less developed than the one for univariate ones. In the classical setting of functions bandlimited to a compact set $\Omega \subset \R^2$, Landau \cite{la67} showed that the density condition 
\begin{align*}
D^{-}(\Lambda) := \liminf_{R\to\infty} \inf_{x\in \mathbb{R}^2} \tfrac{1}{R^2} \# \Lambda \cap [x-R/2,x+R/2]^2 \geq |\Omega|
\end{align*}
is necessary for a set $\Lambda \subset \R^2$ to be a SS for the Paley-Wiener space
\begin{align*}
PW_{\Omega}(\R^2) = \{f \in L^2(\R^2):\, \supp(\hat{f}) \subset \Omega\}.
\end{align*}
Above $|\Omega|$ stands for the measure of $\Omega$. Landau's result holds also for compact sets in $\R^d$.

Simple examples show that no corresponding sufficient condition for sampling can be formulated in terms of Beurling's lower density $D^{-}(\Lambda)$. Nevertheless, when $\Omega$ is a symmetric convex body, Beurling's maximal gap theorem \cite{MR0427956,MR2917231} provides a useful sufficient sampling condition in terms of the covering radius rather than the density of a set $\Lambda$, and is instrumental to construct sampling lattices, and other simple sampling patterns favored by practitioners.

Similarly, sampling sets $\Lambda$ for the space of functions spanned by the shifts of the bivariate Gaussian \eqref{eq_i2} satisfy $D^{-}(\Lambda) \geq 1$ but cannot be described solely in terms of Beurling's lower density $D^{-}(\Lambda)$, even excluding the critical case $D^{-}(\Lambda) = 1$. In fact, it is easy to show that
\begin{align}\label{eq_i6p}
\Lambda = \Gamma_1 \times \Gamma_2 \mbox{ is a SS for }V^2_a(\mathbb{R}^2)
\Longleftrightarrow \Gamma_1 \mbox{ and } \Gamma_2 \mbox{ are SS for } V^2_a(\mathbb{R}).
\end{align}
Thus, by Theorem \ref{th_i3}, for a product set $\Lambda = \Gamma_1 \times \Gamma_2$ with $\Gamma_j \subset \mathbb{R}$ separated, 
\begin{align}\label{eq_i6}
D^-(\Gamma_j) >1, j=1,2 \Rightarrow \Lambda \mbox { SS for }V^2_a(\R^2) \Rightarrow D^-(\Gamma_j) \geq 1, j=1,2.
\end{align}

As a consequence, the set $\Lambda = \frac{1}{\alpha} \mathbb{Z}\times\frac{1}{\beta}\mathbb{Z}$ with $0<\alpha<1$ and $\beta>0$ fails to be a sampling set for $V^2_a(\R^2)$ even though its density
$D^{-}(\Lambda)=\alpha \beta$ can be arbitrary large, if $\beta$ is taken to be large. On the other hand, a squared lattice of the form $\Lambda= (1+\varepsilon)^{-1} \big(\mathbb{Z} \times \mathbb{Z}\big)$ with $\varepsilon>0$ is a sampling set for $V^2_a(\R^2)$ and its density 
$D^{-}(\Lambda) = (1+\varepsilon)^2$ can be arbitrarily close to 1.

Our first result is in the spirit of
\eqref{eq_i6}, and gives conditions for \emph{slanted point configurations} to be sampling sets for $V^2_a(\R^2)$.
\begin{theorem}\label{th_samp}
Let $p,q \in \mathbb{Z}$ be relatively prime and set $\sigma := \sqrt{p^2+q^2}$.
For separated $\Gamma_1, \Gamma_2 \subset \mathbb{R}$ consider the set
\begin{align}\label{eq_lambdaset}
\Lambda := \begin{bmatrix} p/\sigma & - q/\sigma\\
q/\sigma & p/\sigma
\end{bmatrix} \cdot \tfrac{1}{\sigma} \Gamma_1 \times \sigma \Gamma_2.
\end{align}
\begin{itemize}
\item[(a)] If $\Gamma_1, \Gamma_2 \subset \mathbb{R}$ are separated and
\begin{equation}\label{eq_dens_1}
    D^{-}(\Gamma_1) > 1 \quad \text{and} \quad D^{-}(\Gamma_2) > 1 
\end{equation}
or
\begin{equation}\label{eq_dens_2}
    D^{-}(\Gamma_1) > \frac{1}{\sigma^2} \quad \text{and} \quad D^{-}(\Gamma_2) > \sigma^2 
\end{equation}
then $\Lambda$ is a sampling set for $V^2_a(\mathbb{R}^2)$.
\item[(b)] If $\Lambda$ is a sampling set for $V^2_a(\mathbb{R}^2)$ with $\Gamma_1,\Gamma_2$ separated, then
\begin{equation}\label{eq_not_dens}
D^{-}(\Gamma_1) \ge \frac{1}{\sigma^2}, 
\quad D^{-}(\Gamma_2) \ge 1, \quad\mbox{and} \quad D^{-}(\Gamma_1) D^{-}(\Gamma_2) \ge 1.
\end{equation}
\item[(c)] There exists a separated set $\Gamma_1 \subset \R$ with $D^{-}(\Gamma_1) = \frac{1}{\sigma^2}$ such that, for any separated $\Gamma_2 \subset \R$, the set $\Lambda$ is not a sampling set for $V^2_a(\R^2)$. Similarly, there exists a separated set $\Gamma_2 \subset \R$ with $D^{-}(\Gamma_2) = 1$ such that for any separated set $\Gamma_1 \subset \R$, the set $\Lambda$ is not a sampling set for $V^2_a(\R^2)$.
\end{itemize}
\end{theorem}

Some remarks are in order. First, Theorem \ref{th_samp} holds in particular for $(p,q) \in \{(1,0),(0,1)\}$, which recovers \eqref{eq_i6}.
Second, Theorem \ref{th_samp} applies verbatim replacing the space $V^2_a(\R^2)$ by its $L^p$ variant
 \begin{align}\label{eq_vp}
 V^p_a(\mathbb{R}^2):=\Big\{f(x,y)=\sum_{(n,m)\in\mathbb Z^2}c_{n,m}e^{-a(x-n)^2-a(y-m)^2}:c\in l^p(\mathbb Z^2)\Big\},
 \end{align}
 $1 \leq p \leq \infty$,
 provided that the sampling inequalities \eqref{eq_i1} are considered with respect to $L^p$ and $\ell^p$ norms; see Section \ref{sec_p}.

\begin{figure}
 \begin{subfigure}{0.24\textwidth}
     \includegraphics[width=\textwidth]{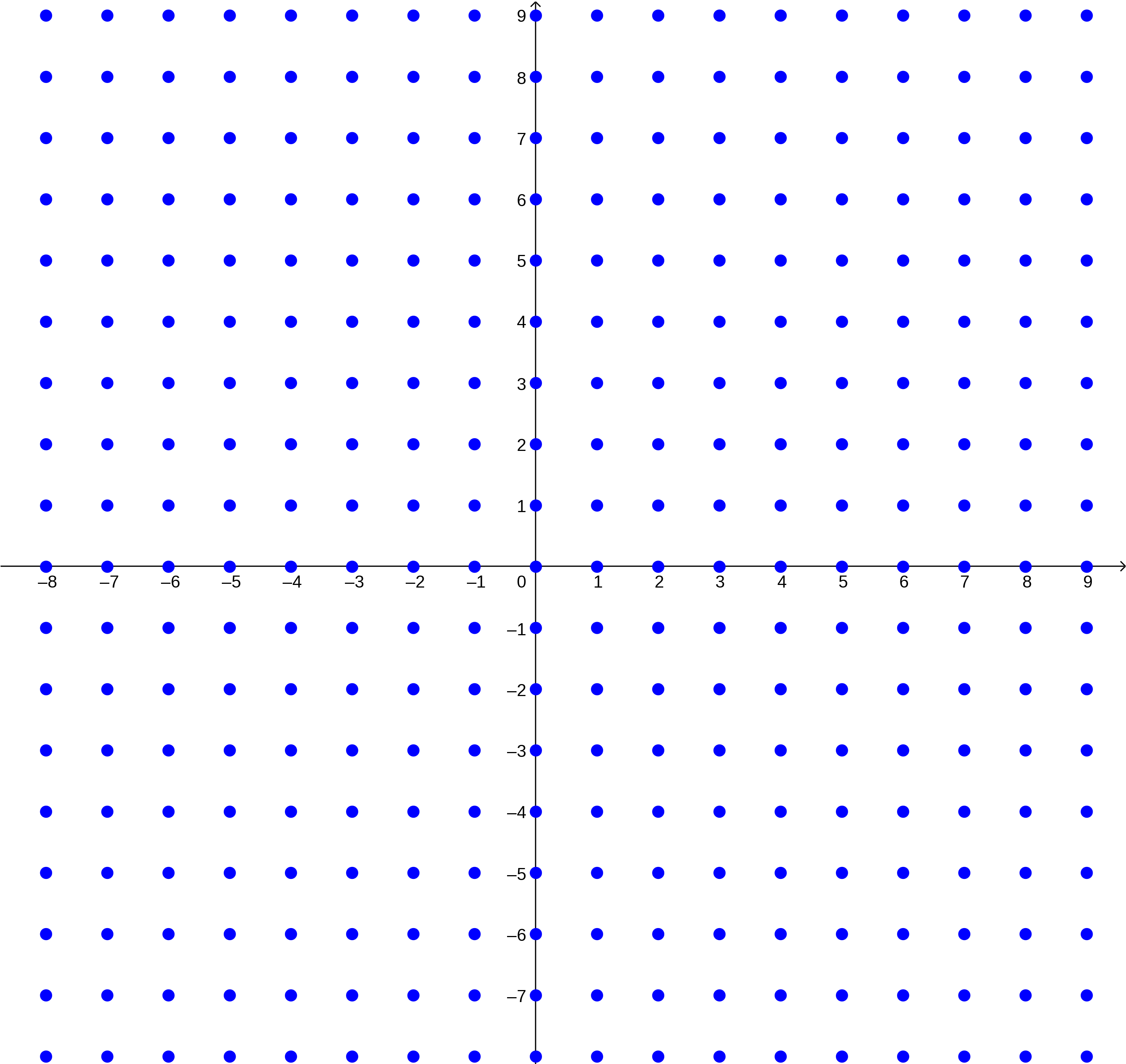}
     \caption{$p=0,\, q=1$}
     \label{fig:a}
 \end{subfigure}
 \hfill
 \begin{subfigure}{0.24\textwidth}
     \includegraphics[width=\textwidth]{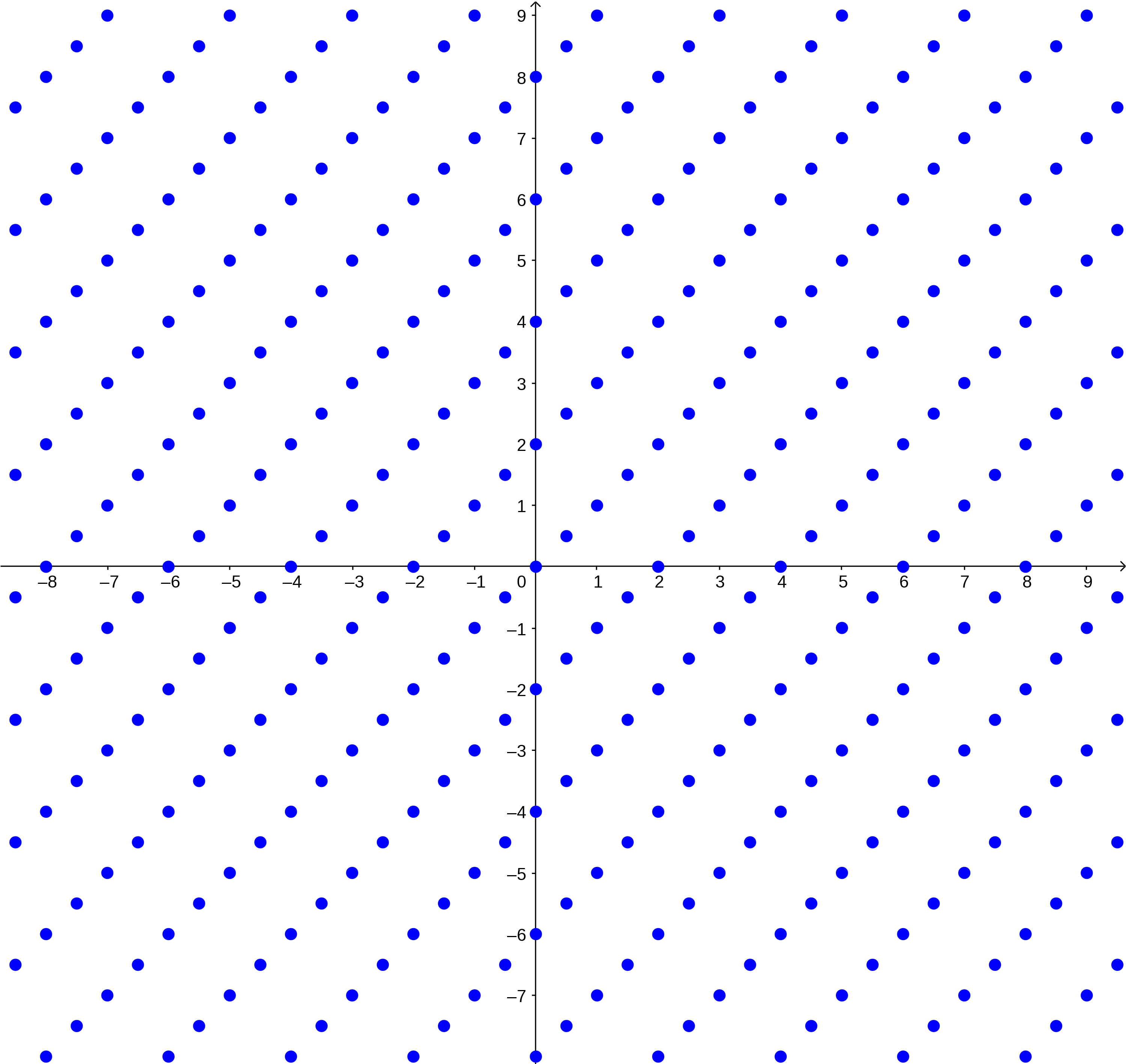}
     \caption{$p=1,\, q=1$}
     \label{fig:b}
 \end{subfigure}
  \hfill
 \begin{subfigure}{0.24\textwidth}
     \includegraphics[width=\textwidth]{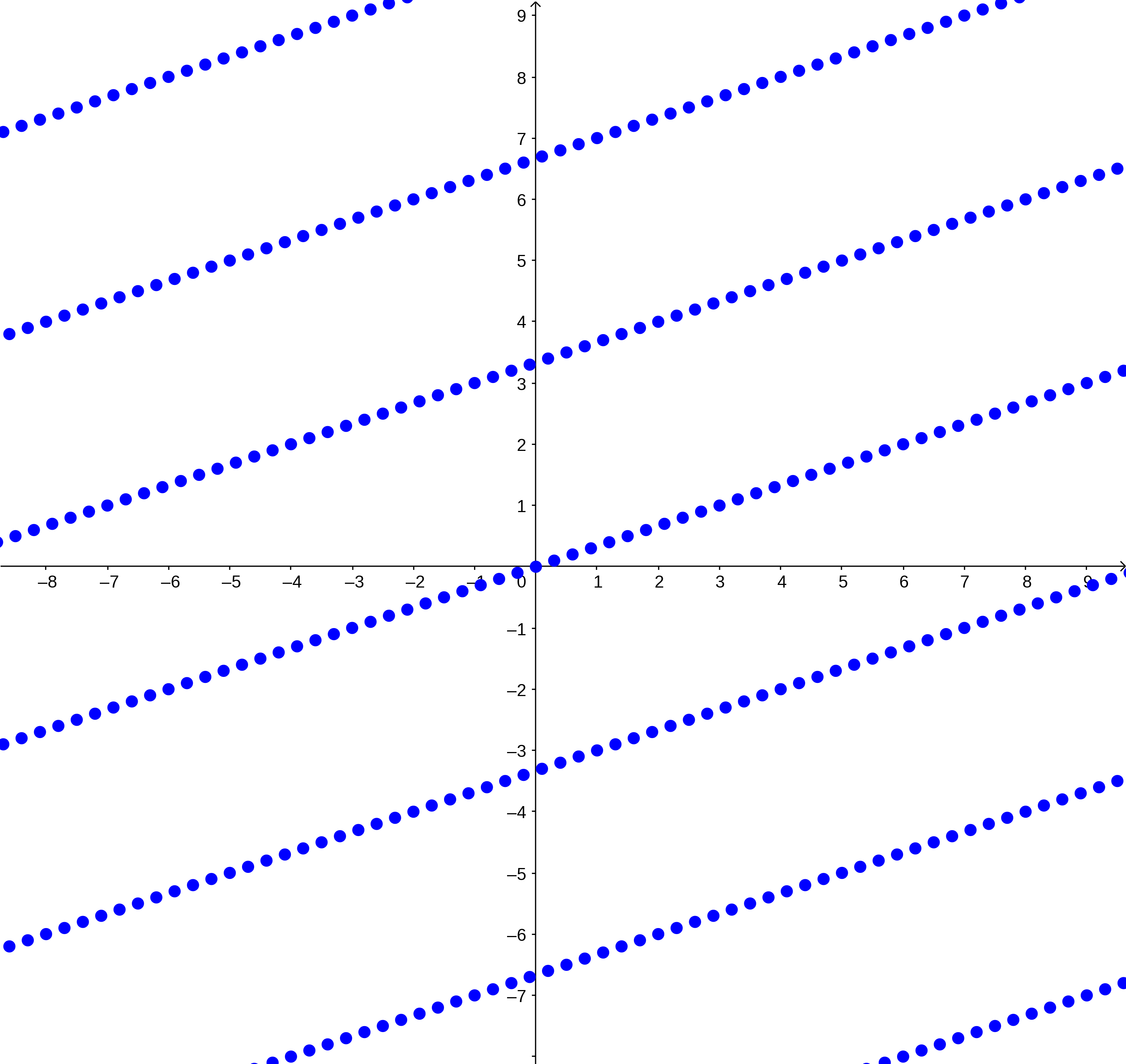}
     \caption{$p=1,\, q=3$}
     \label{fig:c}
 \end{subfigure}
 \hfill
 \begin{subfigure}{0.24\textwidth}
     \includegraphics[width=\textwidth]{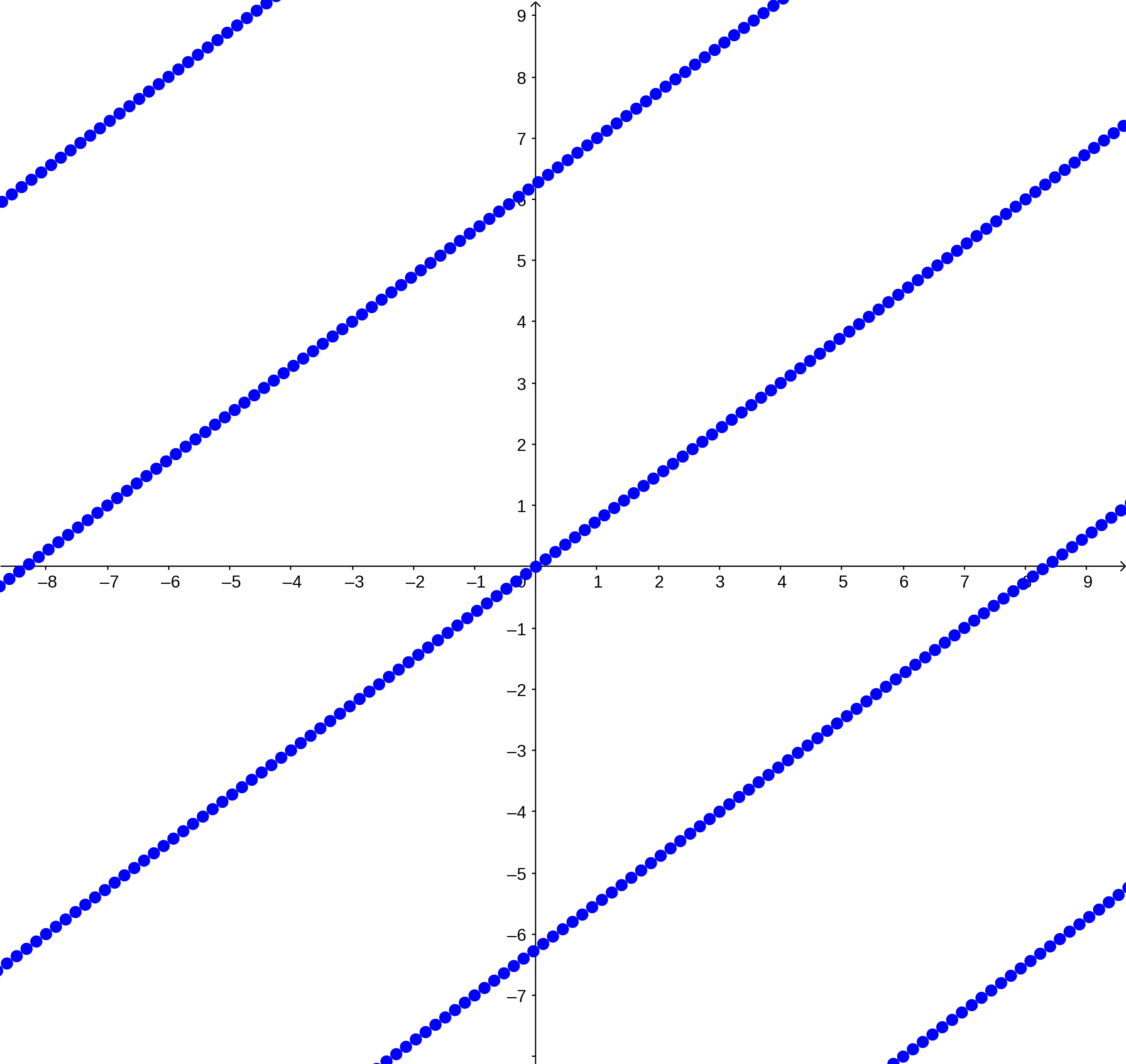}
     \caption{$p=3,\, q=4$}
     \label{fig:d}
 \end{subfigure}
  \caption{The lattice $\Lambda$ for various values $p,q$, and $D^{-}(\Gamma_i) \approx 1$.}
 \label{lattices_pic}
\end{figure}

Choosing in particular $\Gamma_j = \alpha_j \Z$ with $\alpha_j \in (0,1)$, Theorem \ref{th_samp} provides many examples of \emph{sampling lattices} for the space $V^2_a(\R^2)$.

By combining Theorem \ref{th_samp} with a factorization formula from \cite{MR1884237} the result can be extended to shift-invariant spaces generated by other functions besides the Gaussian, such as the tensor product of hyperbolic secants; see Section \ref{sec_other}.

The set \eqref{eq_lambdaset} lies at the intersection of two families of parallel lines, whose slopes depend on $p$ and $q$. Maintaining the densities of $\Gamma_1$ and $\Gamma_2$ fixed and above the value $1$ while increasing $p^2+q^2$ results in a sampling pattern where points are very densely spread along a rather sparse family of parallel lines; see Figure~\ref{lattices_pic}. This hints at a certain possible compensation between the density of the array of lines on the one hand, and the sampling density along each line on the other. To decouple these two factors, we next let the sampling density on each line tend to infinity, while keeping the density of inter-line separations constant, and investigate the sampling problem in this regime.

\subsection{Sampling trajectories}\label{sec_st}

Our second result concerns sampling along collections of parallel lines. In this case, we speak of a \emph{trajectory} $\Lambda \subset \mathbb{R}^2$ and call it a \emph{sampling trajectory} (ST) if there exist constants $A,B>0$ such that
\begin{align}\label{eq_ms}
A \| f \|_2^2 \leq \int_\Lambda |f(s)|^2 \, ds \leq B \|f\|_2^2, \qquad f \in V^2_a(\mathbb{R}^2),
\end{align}
where integration is with respect to arc-length. 

We parametrize families of parallel lines by a unit vector $\vv\in\R^2$ and a set $\Gamma \subset \R$, and define a trajectory $\Lambda=\Lambda(\Gamma)$ by 
\begin{equation}\label{la}
    \Lambda:=\{(x,y)\in\R^2: (x,y)\cdot\vv\in\Gamma\},
\end{equation}
where $(x,y)\cdot \vv$  is the usual dot-product in $\R^2$. 
\begin{figure}[!ht]
\centering
\includegraphics[height=6cm]{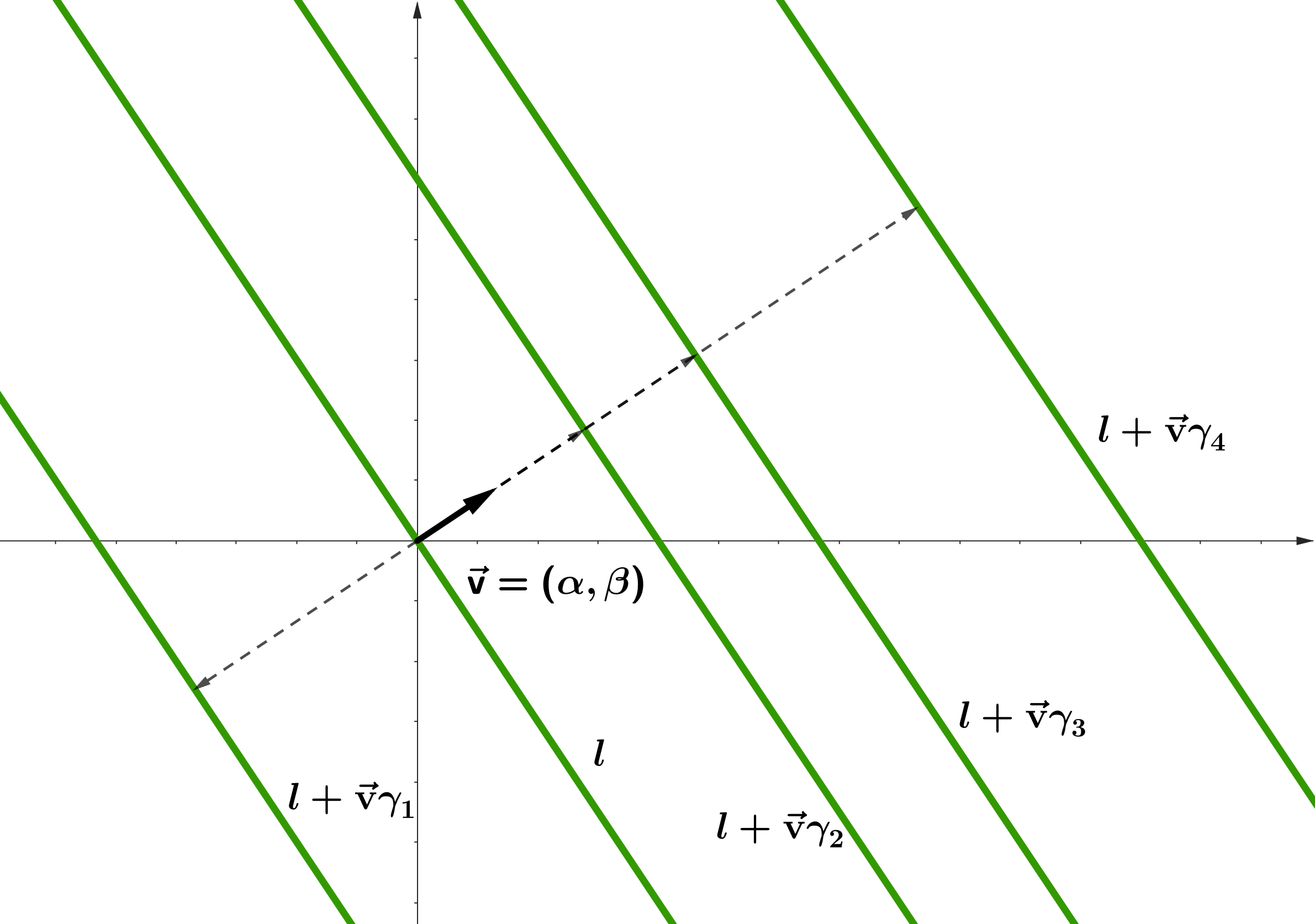}
\caption{Part of the set \eqref{la}; $\gamma_i \in\Gamma$ for $i=1\dots4$}
\label{img1}
\end{figure}
Part of the interest in sampling with lines, and other parametric families of curves, comes from applications where data are acquired by moving sensors \cite{u1,u2}. In that setting, the acquisition cost is better modeled by the density of the \emph{collection} of lines --- in our case parametrized by $\Gamma$ and $\vv$ ---as a proxy for scanning times \cite{u1,u2}, rather than by the density of a discrete sampling pattern contained in $\Lambda$. The challenge is then to characterize the sampling property \eqref{eq_ms} in terms of the intrinsic quality measure $D^-(\Gamma)$ \cite{g, Jaming, RUZ}.

Our main result on sampling trajectories reads as follows.
\begin{theorem}\label{main_result} Let $\Gamma\subset\R$ be a separated set, $\vv=(\alpha,\beta) \in \R^2$ a unit vector and $\La$ the collection of lines \eqref{la}. Set
\begin{align}\label{sigma}
\sigma(\vv)=\left\{ \begin{array}{ll}
\infty&  \mbox{if } \frac{\beta}{\alpha}\not\in\mathbb Q \\ \sqrt{p^2+q^2},& \mbox{if } \beta/\alpha=p/q\in\mathbb Q\cup\{\infty\}, \mbox{gcd}\,(p,q)=1.\end{array}\right.
\end{align}

 \begin{enumerate}
       \item[{\rm(a)}]
 If $\sigma(\vv)=\infty,$ then the condition
\begin{equation}\label{1}D^-(\Gamma) >1/ \sigma(\vv)\end{equation}
is both necessary and sufficient  for $\Lambda$ to be a ST for $V^2_a(\R^2)$.
    \item[{\rm(b)}]
 If $\sigma(\vv)<\infty,$ then \eqref{1} is sufficient for $\Lambda$ to be a ST for $V^2_a(\R^2)$, while 
\begin{equation}\label{1p}D^-(\Gamma) \geq 1/ \sigma(\vv)\end{equation}
 is necessary for $\Lambda$ to be a ST for $V^2_a(\R^2)$.
    \item[{\rm(c)}]
 If $\sigma(\vv)<\infty$ and 
\begin{align*}
D^+(\Gamma) := \limsup_{R\to\infty} \sup_{x\in \mathbb{R}} \tfrac{1}{R} \# \Gamma \cap [x-R/2,x+R/2]
<1/\sigma(\vv)
\end{align*}
then there is a non-trivial function $f\in V_a^1(\R^2)$ which vanishes on $\La$, and so $\La$ is not a ST for $V^2_a(\R^2)$.
\end{enumerate}
{\rm(}In every case, we interpret $1/\sigma(\vv)=0$ if $\sigma(\vv)=\infty$.{\rm)}
\end{theorem}
The space $V_a^1(\R^2)$ mentioned in part (c) of the theorem is defined in \eqref{eq_vp} and is a proper subset of $V_a^2(\R^2)$. As with Theorem \ref{th_samp}, with a suitable formulation of the sampling inequalities, the same claims hold with respect to $L^p$ norms, $1\leq p \leq\infty$; see Section \ref{sec_p}.

To compare, we quote the following result for the Paley-Wiener space, which is a consequence of \cite[Theorems 2 and 3]{RUZ}.
\begin{theorem}\label{th_ruz}
Let $\Omega \subset \R^2$ be compact, convex and with positive measure, $\Gamma\subset\R$ separated, $\vv \in \R^2$ a unit vector, and $\La$ the collection of lines \eqref{la}.
\begin{itemize}
\item[(a)] If $D^-(\Gamma) \vv \notin \Omega - \Omega$, then $\Lambda$ is a ST for $PW_\Omega(\R^2)$.
\item[(b)] If $D^-(\Gamma) \vv$ belongs to the interior of $\Omega - \Omega$, then $\Lambda$ is not a ST for $PW_\Omega(\R^2)$. 
\end{itemize}
\end{theorem}
For lines with a rational slope, i.e., $\sigma(\vv)<\infty$,
Theorem \ref{main_result} is consistent with Theorem \ref{th_samp}. On the other hand, for families with an irrational slope, i.e., $\sigma(\vv)=\infty$, we see a surprising contrast between Theorems \ref{main_result} and
\ref{th_ruz}: while in the Paley-Wiener space the sampling property requires the average separation between parallel lines not to exceed a certain critical threshold, for the shift-invariant space \eqref{eq_i1} \emph{any} progression of distances between consecutive lines is allowed, as long as these remain bounded. This is all the more remarkable because in the one dimensional case, the density-sampling results in the Paley-Wiener and Gaussian shift-invariant space (Theorem \ref{th_i3}) are identical.

\subsection{Gabor frames}
The sampling problem in shift-invariant spaces is intimately connected to the spanning properties of \emph{Gabor systems}, which are functional dictionaries produced by translation and modulation of a function $g \in L^2(\R^2)$:
\begin{equation*}
\pi(t_1,t_2,\omega_1,\omega_2)g(x,y) = e^{2 \pi i x \omega_1} e^{2 \pi i y \omega_2} g(x - t_1, y - t_2), \qquad (t_1,t_2,\omega_1,\omega_2) \in \R^4.
\end{equation*}
The Gabor system generated by the \emph{window function} $g$ along a \emph{set of time-frequency nodes} $\Delta \subset \R^4$ is
\begin{align}
\mathcal{G}(g,\Delta) = \{ \pi(t_1,t_2,\omega_1,\omega_2)g : (t_1,t_2,\omega_1,\omega_2) \in \Delta\},
\end{align}
and the main question is whether it forms a  \emph{frame} of $L^2(\R^2)$, that is, whether there exist positive constants $A,B$ such that
\begin{equation}
     A\|f\|^2_{2} \le \sum\limits_{z \in \Delta} \left| \langle \pi(z)g, f \rangle \right|^2
     \le B \|f\|^2_2, \qquad f \in L^2(\R^2).
 \end{equation}
Gabor systems are much better understood in dimension one. The frame property is completely characterized for the univariate Gaussian window \cite{MR1188007,MR1173118}, whereas for other special windows, such as certain totally positive functions and certain rational functions, all sampling \emph{lattices} are classified in terms of their volume \cite{grs,MR3053565,MR1884237,MR4542702}.

Results in dimension two or more are scarcer. While time-frequency lattices leading to frames need to have volume $<1$, this condition is far from sufficient even for Gaussian windows; see \cite{GL} for various explicit counterexamples. On the other hand a tensor product argument provides numerous examples of \emph{complex} lattices with any desired volume satisfying the basic restriction $<1$ that provide Gabor frames for the standard Gaussian function
\begin{align}\label{gab_pi}
g_\pi(x,y) = e^{-\pi(x^2+y^2)}, \qquad (x,y) \in \R^2.
\end{align}
Here, we call a lattice $\Delta \subset \R^4$ complex, if, under the identification $\R^2 \times \R^2 \simeq \mathbb{C}^2$, it satisfies $\mathrm{i} \Delta = \Delta$. Under this condition, certain symmetries of the Bargmann-Fock space of analytic functions can be brought to bear on the Gabor frame problem, an approach that was championed in \cite{g_gabor}, where the question of comparable results for non-complex lattices was raised. In addition, lattices $\Lambda \subset \R^4$
with volume $<1/2$ lead to Gabor frames for Gaussian windows, provided
that they satisfy a certain genericity condition known as \emph{transcendentality} \cite{Luef_gabor}.

An example of a non-complex lattice yielding a Gabor frame for the bivariate Gaussian window \eqref{gab_pi} was found in \cite{PR_gabor}. The Gabor system $\mathcal{G}(g_{\pi}, \Delta)$ with lattice
\begin{align}\label{eq_lpr}
\Delta = \Z^2 \times \begin{bmatrix}
    a & a\\
    -b & b  
\end{bmatrix} \Z^2
\end{align}
 is a frame of $L^2(\R^2)$ if $0<a<1/2$ and $0<b<1/2$, while it fails to be so if $a>1/2$ and $b>1/2$ \cite{PR_gabor}. This example is genuinely different from the ones covered in \cite{g_gabor}: while an adequate symplectic transformation maps the lattice \eqref{eq_lpr} into a complex one, in order to preserve the frame property, the Gaussian function \eqref{gab_pi} needs be correspondingly transformed into a non-isotropic Gaussian function, which is out of the scope of \cite{g_gabor}.

As an application of our sampling results, we shall obtain the following.
\begin{theorem}\label{Gabor_thm}
    Let $\Lambda \subset \R^2$ admit the representation~\eqref{eq_lambdaset} with $\Gamma_j\subset \R$ separated and let $g_a(x,y)=e^{-a(x^2 +y^2)}$ with $a>0$. 
    \begin{enumerate}
        \item[(a)] If the densities of the sets $\Gamma_1$ and $\Gamma_2$ satisfy either~\eqref{eq_dens_1} or ~\eqref{eq_dens_2}, then the Gabor system
    $\mathcal{G}(g_a, \Lambda \times \Z^2)$ is a frame of $L^2(\R^2)$.  
\item[(b)] If $\mathcal{G}(g_a, \Lambda \times \Z^2)$ is a frame of $L^2(\R^2)$ then
\begin{equation}\label{eq_not_dens_2}
D^{-}(\Gamma_1) > \frac{1}{\sigma^2}, 
\quad D^{-}(\Gamma_2) > 1, \quad\mbox{and} \quad D^{-}(\Gamma_1) D^{-}(\Gamma_2) > 1.
\end{equation}
    \end{enumerate}
\end{theorem}
As an application of Theorem \ref{Gabor_thm} we obtain the following corollary for 
anisotropic Gaussian windows \[g_{\alpha,\beta}(x,y) := e^{-(\alpha x^2+\beta y^2)}, \qquad x,y\in\R,\] with shape parameters $\alpha,\beta>0$ and time-frequency shifts given by lattices.

\begin{corollary}\label{coro_gabor}
Let $p,q \in \Z$ be relatively prime and $\sigma:=\sqrt{p^2+q^2}$.
Consider the time-frequency lattice
\begin{align*}
\Delta_{a,b,c,d} := \Bigg(
\begin{bmatrix} p/(a\sigma) & - q/(a\sigma)\\
q/(b\sigma) & p/(b\sigma)
\end{bmatrix} \cdot \tfrac{c}{\sigma} \Z \times d\sigma \Z \Bigg) \times a\Z \times b\Z,
\end{align*}
with $a,b,c,d >0$, and the anisotropic Gaussian function
with shape parameters $\alpha>0$ and $\beta:=\alpha \tfrac{b^2}{a^2}$.

If either
{\rm(}i{\rm)} $c<1$ and $d<1$, or {\rm(}ii{\rm)} $c<\sigma^2$ and $d < \sigma^{-2}$, then $\mathcal{G}\big(g_{\alpha, \beta}, \Delta_{a,b,c,d}\big)$ is a frame of $L^2(\R^2)$.
\end{corollary}
In the isotropic case $a=b$, Corollary \ref{coro_gabor} provides a large collection of non-complex lattices yielding Gabor frames for the standard Gaussian window, which includes \eqref{eq_lpr}, and partially answers a question raised in \cite{g_gabor}. Theorem \ref{Gabor_thm} also complements \cite{Luef_gabor} by providing many examples of Gabor frames with \emph{non-transcendental (adjoint) lattices} (and, indeed, note that the lattice $\Delta_{a,b,c,d}$ in the theorem can have volume close to 1, while it is conjectured in \cite{Luef_gabor} that lattices with transcendental adjoints cannot have volume $>1/2$ is if they lead to Gabor frames.)

\subsection{Organization}
Section \ref{sec_p} provides definitions and background results.
In Sections \ref{sec_1} and \ref{sec_2} we study a number of \emph{uniqueness problems} for the shift-invariant space \eqref{eq_i2} and other related spaces, that is, we investigate which sets are contained in zero sets of non-zero functions. To this end, we study the restriction of bivariate functions in the shift-invariant space to lines, and aptly invoke one dimensional results together with arguments of congruence and almost periodicity. Theorems \ref{main_result}
and \ref{th_samp} are then proved in Sections \ref{sec_3} and \ref{sec_4} respectively by resorting to a compactness technique going back to Beurling. Theorem \ref{Gabor_thm} and Corollary \ref{coro_gabor} are proved in Section \ref{sec_5}, while Section \ref{sec_other} discusses extensions of the results to shift-invariant spaces generated by window functions closely related to the Gaussian. The proofs of certain auxiliary results are postponed to Section \ref{sec_pos}.

\section{Preliminaries}\label{sec_p}
For two functions $f,g:X \to [0,\infty)$ we write $f \lesssim g$ if there exists a constant $C\geq 0$ such that $f(x) \leq C g(x)$, for all $x \in X$. We also write $f \asymp g$ if $f \lesssim g$ and $g \lesssim f$. 

To prove sampling theorems, we will use a technique of Beurling that involves considering different $L^p$-norms and a certain compactness property. We review the necessary notions.

We denote the $L^p$-norm of a function $f:\R^d \to \mathbb{C}$ by $\|f\|_p$, while the $\ell^p$-norm of its restrictions to a discrete set $\Delta\subset\R^d$ is
$\|f|_\Delta\|_p^p:=\sum_{u \in\Delta}|f(u)|^p$, with the usual modifications for $p=\infty$. For the set of lines \eqref{la} we write
$$\|f|_\La\|_p^p:=\int_\La|f(u,v)|^p\,ds,$$ where integration is with respect to arc-length and $p<\infty$. We also write $\|f|_\La\|_\infty := \sup_{z\in\La} |f(z)|$, and will only do so for continuous functions $f$, so there are no concerns about null-measure sets.

We shall be interested in the shift-invariant spaces with Gaussian generator 
\begin{align}\label{eq_vpa2}
V^p_a(\R^d) = \Bigg\{ \sum_{n \in \Z^d} c_n g_a(\cdot-n): c \in \ell^p(\Z^d)\Bigg\},
\end{align}
where $d=1,2$, $1 \leq p \leq \infty$ and $g_a(x)=e^{-a|x|^2}$. The series in \eqref{eq_vpa2} converge in the $L^p$-norm for $p<\infty$ and in the $\sigma(L^\infty,L^1)$-topology for $p=\infty$; in addition, for all $1 \leq p \leq \infty$, the coefficients are unique, $\|f\|_p \asymp \|c\|_p$, and the series converges uniformly on compact sets \cite{ag01}. We say that $\Delta$ is a sampling set (SS) for $V^p_a(\R^d)$ if there exist positive constants $A,B$ such that
\begin{align*}
A \|f\|_p \leq \|f|_\Delta\|_p \leq B \|f\|_p, \qquad f \in V^p_a(\R^d),
\end{align*}
while $\Delta$ is a uniqueness set (US) for $V^p_a(\R^d)$ if the only function $f \in V^p_a(\R^d)$ that vanishes on $\Delta$ is the zero function. Similarly, we say that $\Lambda$, given by \eqref{la}, is a sampling trajectory (ST) for $V^p_a(\R^d)$ if there exist positive constants $A,B$ such that
\begin{align*}
A \|f\|_p \leq \|f|_\La\|_p \leq B \|f\|_p, \qquad f \in V^p_a(\R^d).
\end{align*}

The lower and upper Beurling densities of a set $\Delta \subset \R^d$ are
\begin{align}
D^-(\Delta) &= \liminf_{R\to\infty} \inf_{x\in \mathbb{R}} \tfrac{1}{R^d} \# \Delta \cap [x-R/2,x+R/2]^d,\\\label{dplus}
D^+(\Delta) &= \limsup_{R\to\infty} \sup_{x\in \mathbb{R}} \tfrac{1}{R^d} \# \Delta \cap [x-R/2,x+R/2]^d.
\end{align}
The condition $D^-(\Delta)\geq 1$ is necessary for $\Delta$ to be a SS for $V^p(\R^d)$. This follows from an adaptation of Landau's work on Paley-Wiener spaces \cite{la67}, for example in the abstract formulation of \cite{MR2224392,MR3742438}.

Recall that $\Delta \subset \R^d$
is called \emph{separated} if its separation constant 
\eqref{sep_const} is positive: $\delta(\Delta)>0$.
Let $\Delta_n\subset\R^d $ be a sequence of sets satisfying $\inf_n \delta(\Delta_n)>0$. We say that the sequence \emph{converges weakly} to a set $\Delta'\subset\R^d$ if  for every $\epsilon>0$ and $R>0$ there is an integer $l=l(\epsilon,R)$ such that
$$\Delta_k\cap(-R,R)^d\subset \Delta'+(-\epsilon, \epsilon)^d,\quad \Delta'\cap(-R,R)^d\subset \Delta_k+(-\epsilon, \epsilon)^d, \quad k\geq l. $$

For separated $\Delta \subset \R^d$ and arbitrary $\Sigma \subset \R^d$ we let $W_{\Sigma}(\Delta)$ denote the collection of all possible weak limits of the $\Sigma$-translates $\Delta+x_n, n\to\infty$, where $x_n\in\Sigma$.

We will use the following characterization of sampling sets, which goes back to Beurling's work on bandlimited functions \cite{MR10576141a}. For definitedness we quote the following special case of \cite[Theorem 3.1]{grs}. \footnote{The quoted result is formulated in dimension $d=1$ but the argument is valid for any $d$.}

\begin{proposition}\label{prop_s1}
Let $\Delta \subset \R^d$ be separated, $d=1,2$, and $a>0$.
Then the following are equivalent.
\begin{itemize}
\item[(a)] $\Delta$ is a sampling set for $V^p_a(\R^d)$ for some $p \in [1,\infty]$.
\item[(b)] $\Delta$ is a sampling set for $V^p_a(\R^d)$ for all $p \in [1,\infty]$.
\item[(c)] Every $\Delta' \in W_{\Z^d}(\Delta)$ is a sampling set for $V^\infty_a(\R^d)$.
\item[(d)] Every $\Delta' \in W_{\Z^d}(\Delta)$ is a uniqueness set for $V^\infty_a(\R^d)$.
\end{itemize}
\end{proposition}
Similarly, we will use the following characterization of sampling trajectories.
\begin{proposition}\label{prop_s3}
Let $\Gamma \subset \R$ be separated and $a>0$. 
Let $\Lambda$ be the collection of lines with direction vector $\vv$ \eqref{la} and set 
\begin{equation}\label{sigma_set}
\Sigma := \vv \cdot \mathbb{Z}^2 = \{\vv\cdot(n,m): (n,m) \in \Z^2\} \subset \R.  
\end{equation}
Then the following are equivalent.
\begin{itemize}
\item[(a)] $\Lambda$ is a ST for $V^p_a(\R^2)$ for some $p \in [1,\infty]$.
\item[(b)] $\Lambda$ is a ST for $V^p_a(\R^2)$ for all $p \in [1,\infty]$.
\item[(c)] For every $\Gamma' \in W_{\Sigma}(\Gamma)$, the collection of lines
\begin{equation}\label{la2}
    \Lambda':=\{(x,y)\in\R^2: (x,y)\cdot\vv\in\Gamma'\}
\end{equation}
is a ST for $V^\infty_a(\R^2)$.
\item[(d)] For every $\Gamma' \in W_{\Sigma}(\Gamma)$, the collection of lines \eqref{la2} is a uniqueness set for $V^\infty_a(\R^2)$.
\end{itemize}
\end{proposition}
Proposition \ref{prop_s3} follows from Proposition \ref{prop_s1} by a discretization argument, whose proof is postponed to Section \ref{sec_pos}.

\section{Uniqueness along parallel lines}\label{sec_1}
As a first step towards Theorem \ref{main_result}, we show that certain lines are uniqueness sets for $V^\infty_a(\R^2)$.
\begin{theorem}\label{t_uniq}
Let $\vv=(\alpha,\beta) \in \R^2$ be a unit vector and $\sigma(\vv)$ be given by \eqref{sigma}.
\begin{enumerate}
    \item[{\rm(i)}] If $\sigma(\vv)=\infty$ then every
    line $$\{(x,y)\in\R^2: (x,y)\cdot \vv=\gamma\},\quad \gamma\in\R,$$ is a US for $V^\infty_a(\R^2)$.
    \item[{\rm(ii)}] If  $\sigma(\vv)<\infty$, and $\Gamma \subset \R$ is separated and satisfies \eqref{1}, then the set $\La$ in \eqref{la} is a US for $V^\infty_a(\R^2)$.
\end{enumerate}
\end{theorem}
\begin{proof}

The argument below does not depend on $a,$ so for simplicity we assume that $a=1.$ 

(i) Take any $\gamma \in \R$ and assume that $f\in V_a^\infty(\R^2)$ vanishes on the line 
$$
l = 
\{\beta y + \alpha x = \gamma, \, (x,y) \in \R^2\}.$$ We will show that $f = 0$.

Write $f(x,y)=\sum_{n,m\in\Z} c_{n,m} e^{-(x-n)^2 - (y-m)^2}$ with $c \in \ell^\infty(\Z^2)$,
and extend it analytically to $\Cc\times\Cc$ by
\begin{align*}
f(z,w)=\sum_{n,m\in\Z} c_{n,m} e^{-(z-n)^2 - (w-m)^2}
=\sum_{n,m\in\Z} c_{n,m} e^{-n^2-m^2} e^{2nz-z^2+2mw-w^2}.
\end{align*}
Since $c\in \ell^\infty(\Z^2)$ the series is easily seen to be uniformly and absolutely convergent on compact subsets of $\Cc\times\Cc$.

Set $\tau = -\alpha/\beta$ and $\gamma' = \gamma/ \beta$. Note that $\tau \notin \Q$, since $\sigma(\vv) = \infty$. 

The condition $f|_l=0$ is equivalent to the condition  $f(x,\tau x+\gamma')=0, x\in\R$. Since the function $f(z,\tau z+\gamma')$ is entire, it vanishes identically. This gives
\begin{align*}
0 &= f(z, \tau z + \gamma') =  \sum\limits_{n,m} c_{n,m}e^{-(z-n)^2 - (\tau z + \gamma' - m)^2}
\\
&=e^{-z^2(1+\tau^2)}\cdot \sum_{n,m} c_{n,m} e^{-n^2-(\gamma' - m)^2}e^{-2z(\tau \gamma'  - \tau m - n)}, \qquad z\in\Cc. 
\end{align*}

Set $J = J(n,m) = \tau \gamma'  - \tau m - n$. Since   $\tau \not \in \Q$, we have $J(n_1,m_1) \neq J(n,m)$ if $(n_1,m_1) \neq (n,m)$. Therefore, setting $\tilde c_{J} = \tilde c_{J(n,m)} = c_{n,m} e^{-n^2-(\gamma' - m)^2}$ we conclude that
$$
\sum_{J} \tilde c_J e^{-2z J} = 0, \quad z \in \Cc. 
$$Clearly, $\{\tilde c_{J}\} \in l^1(\Z^2)$.
When $z=iy, y\in\R,$  the series above is  the Fourier series of an almost periodic function (see e.g. \cite{bes} or \cite{levitan}). Since this series  is identically zero,  
by the uniqueness theorem for almost periodic functions, we conclude that  $\tilde{c}_J = 0$, whence $c_{n,m} = 0$ for every $(n,m) \in \Z^2$, and so $f = 0$.

(ii) Assume that a function $f\in V_a^\infty$ vanishes on $\Lambda$, and let us show that $f=0$. As before, we extend $f$ analytically to $\Cc\times\Cc$.

We write $\vv=(\alpha,\beta)=(p/\sqrt{p^2+q^2},q/\sqrt{p^2+q^2})$, with  $p,q\in\Z$ and $\mathrm{gcd}(p,q)=1$. Let us assume that $p,q\ne0.$
It is easy to check that the condition  $f|_\Lambda=0$ is equivalent to $$f\left(\frac{\sigma\gamma}{2p}-qz,\frac{\sigma\gamma}{2q}+pz \right)=0,\quad z\in\mathbb C, \gamma\in\Gamma.$$

Hence,
\begin{align*}
0&=f\left(\frac{\sigma\gamma}{2p}-qz,\frac{\sigma\gamma}{2q}+pz\right)
\\
&=\sum_{n,m}c_{n,m}\exp\left\{-\left(\frac{\sigma\gamma}{2p}-qz-n\right)^2-\left(\frac{\sigma\gamma}{2q}+pz-m\right)^2\right\}
\\
&=\exp\left\{-z^2\sigma^2+z\sigma\gamma(q/p-p/q)\right\}\times
\\
&\qquad\qquad\sum_{n,m}c_{n,m}\exp\left\{2z(pm-qn)-\left(\frac{\sigma\gamma}{2p}-n\right)^2-\left(\frac{\sigma\gamma}{2q}-m\right)^2\right\}.
\end{align*}
It follows that the last sum vanishes for all $z$ and $\gamma\in\Gamma.$

For every $k\in\mathbb Z$, choose any $n_0=n_0(k),m_0=m_0(k)$ satisfying 
$pm_0-qn_0=k.$
Then $pm-qn=k$ iff $n=n_0+pl,m=m_0+ql, l\in \mathbb Z,$ and we have
$$\sum_{k=-\infty}^\infty e^{2kz}\sum_{l=-\infty}^\infty c_{n_0+pl,m_0+ql} e^{-(\frac{\sigma\gamma}{2p}-n_0-pl)^2-(\frac{\sigma\gamma}{2q}-m_0-ql)^2}=0,$$
for every $z$ and $\gamma\in\Gamma.$ This gives
\begin{equation}\label{2}\sum_{l=-\infty}^\infty c_{n_0+pl,m_0+ql} e^{-(\frac{\sigma\gamma}{2p}-n_0-pl)^2-(\frac{\sigma\gamma}{2q}-m_0-ql)^2}=0,\end{equation}
for every $n_0,m_0$ and every $\gamma\in\Gamma.$

We note that
\begin{align*}
&(\frac{\sigma\gamma}{2p}-n_0-pl)^2+(\frac{\sigma\gamma}{2q}-m_0-ql)^2
\\
&\quad=
(\frac{\sigma\gamma}{2p}-n_0)^2+(\frac{\sigma\gamma}{2q}-m_0)^2-2l\sigma(\gamma-\frac{n_0p+m_0q}{\sigma})+l^2\sigma^2
\\
&\quad=
(\frac{\sigma\gamma}{2p}-n_0)^2+(\frac{\sigma\gamma}{2q}-m_0)^2-(\gamma-\frac{n_0p+m_0q}{\sigma})^2+\sigma^2((\gamma/\sigma-\frac{n_0p+m_0q}{\sigma^2})- l)^2.
\end{align*}
 Since the first three terms do not depend on $l,$ by \eqref{2} we see that  the function 
$$g_k(\zeta):=\sum_{l=-\infty}^\infty a_l e^{-\sigma^2(\zeta-  l)^2},\quad a_l:=c_{n_0+pl,m_0+ql},$$
vanishes on the set $ \zeta\in \Gamma/\sigma-\frac{n_0p+m_0q}{\sigma^2}.$

By \eqref{1}, for any constant $C$, we have $D^-(\Gamma/\sigma-C)=\sigma D^-(\Gamma)>1$.
By Theorem \ref{th_i3} and Proposition \ref{prop_s1}, the function $g_k$ above is zero for every $k$ (i.e. for every $n_0,m_0$) and so all coefficients $c_{n,m}=0$ and $f=0.$

If one of the numbers $p,q$ is equal to zero, the result can be proved similarly (and somewhat more easily).
\end{proof}
\begin{remark}
Part (ii) of Theorem \ref{t_uniq} holds also for non-separated $\Gamma$, as can be seen by invoking \cite[Theorem 4.4, part (a)]{grs} instead of Theorem \ref{th_i3} in the previous proof.
\end{remark}

\section{Non-uniqueness along Parallel Lines}\label{sec_2}
We start by observing that sampling sets must have positive densities.
\begin{lemma}\label{l2}
Assume that $D^-(\Gamma)=0$, let $\vv \in \R^2$ be a unit vector and let $\Lambda$ be the collection of parallel lines with direction vector $\vv$ given by \eqref{la}. Then for any $p, 1 \le p \leq \infty,$ $\Lambda$ is not a ST for $V^p_a(\R^2)$. 
\end{lemma}
\begin{proof}
We invoke Proposition \ref{prop_s3}.
If $D^-(\Gamma)=0$, then $\Gamma' :=\emptyset \in W_\Sigma(\Gamma)$, where $\Sigma$ was defined in~\eqref{sigma_set}. The corresponding set $\Lambda'$ is also empty, and thus not a US for $V^\infty_a(\R^2)$.
\end{proof}

Second, we provide an explicit condition for non-uniqueness in the univariate shift-invariant space \eqref{eq_i2p}, which may be of independent interest, and is closely related to the decomposition of $V^2_a(\R)$ in terms of so-called small Fock spaces \cite{MR4456797}.

\begin{proposition}\label{prop_1d}
Let $a>0$ and $\Gamma\subset\mathbb R$ be a separated set. Assume that 
\begin{equation}\label{gp}\max\{\#\Gamma\cap[0,r),\#\Gamma\cap(-r,0)\}\leq \rho r+K, \quad r>0,\end{equation}for some $K>0$ and $0<\rho<1$. Then there is a non-trivial a function $f \in V^1_a(\R)$
such that \[\Gamma=\{x \in \mathbb{R}: f(x)=0\}.\]
\end{proposition}
\begin{proof}
 Throughout the proof, we denote by $C$ different positive constants. Set
$$g(z):=\prod_{\gamma\in\Gamma, \gamma\geq0}\left(1-e^{2a(z-\gamma)}\right)\prod_{\gamma\in\Gamma,\gamma<0}\left(1-e^{-2a(z-\gamma)}\right)=:g_+(z)g_-(z).$$
Since $\Gamma$ is separated, each product above converges uniformly on compacts in the complex plane, and so $g_+$ and $g_-$ are entire functions. Clearly, $\Gamma$ is the zero set of $g$.

One can  easily check that the following inequalities hold true:
\begin{equation}\label{e}
    |g_+(z)|<C, \ x\leq 0, \quad |g_-(z)|<C, \ x\geq0,\end{equation} and $$ \left|\prod_{\gamma \geq x}\left(1-e^{2a(z-\gamma)}\right)\right|<C, \ z=x+iy , x\geq0. $$

Let us estimate $g_+$ from above. Set $$n_+(r)=\#\Gamma\cap[0,r).$$Using the last estimate, we get
$$|g_+(x+iy)|\leq C\prod_{0\leq\gamma< x}\left(1+e^{2a(x-\gamma)}\right)\leq C\prod_{0\leq\gamma< x}\left( 2e^{2a(x-\gamma)}\right)\leq$$
$$C2^{n_+(x)}\exp\{2a(xn_+(x)-\sum_{0\leq\gamma< x}\gamma)\}.$$Since
$$\sum_{0\leq\gamma< x}\gamma=\int_{[0,x)}td n_+(t)=xn_+(x)-\int_0^x n_+(x)\,dx,$$by \eqref{gp},  we get
$$|g_+(z)|\leq C e^{a(1-\epsilon)x^2},\quad z=x+iy,\ x\geq0,$$
where $\epsilon$ is any positive number satisfying $\epsilon<1-\rho$ and $C=C(\epsilon,a).$

Similarly, we get the estimate
$$|g_-(z)|\leq C e^{a(1-\epsilon)x^2},\quad x\leq0.$$ This  and \eqref{e} imply
$$|g(z)|\leq Ce^{a(1-\epsilon)x^2},\quad z\in\mathbb C.$$

Set $w=e^{2az}$, and so $x=\log |w|/2a$.  Set
$$h(w):=\prod_{\gamma\in\Gamma, \gamma\geq0}\left(1-we^{-2a\gamma}\right)\prod_{\gamma\in\Gamma,\gamma<0}\left(1-\frac{e^{2a\gamma}}{w}\right).$$
Then $g(z)=h(e^{2az})$. The first product above is an entire function while the second one is analytic for $|w|>0$. Therefore, $h$ is analytic for $|w|>0$ and so can be written as a Laurent series
$$h(w)=\sum_{k=-\infty}^\infty b_kw^k,\quad w\in\Cc, \ w\ne0.$$ By Cauchy-Hadamard's formula,
$$\limsup_{k\to\pm\infty}|b_k|^{1/|k|}=0,$$so that $b_k$ have a superexponential decay  as $k\to\pm\infty$.

For every $|w|\geq 1$ we obtain
$$\left|\sum_{k\geq0}b_kw^k\right|\leq |g(w)|+C\leq Ce^{(1-\epsilon)\log^2|w|/4a}.$$
By Cauchy's inequality, for every $R>0$,
$$|b_k|\leq C\frac{\exp\{(1-\epsilon)\log^2 R/4a\}}{R^k}.$$Letting $R=\exp(2ak/(1-\epsilon))$, we get
$$|b_k|\leq C\exp\{-ak^2/(1-\epsilon)\}, \quad k\geq 0.$$Similarly, one may check that the same estimate holds for every $k<0.$ This shows that the function
$$f(z):=e^{-az^2}g(z)=e^{-az^2}h(e^{2az})$$
admits the representation 
\begin{equation}\label{f}f(z)=\sum_{k\in\mathbb Z}c_k e^{-a(z-k)^2},\quad \{c_k\}=\{b_ke^{ak^2}\}\in l^1(\Z).\end{equation}
Hence, $f\in V_a^1(\R)$. Since $g$ vanishes on $\Gamma,$ so does $f.$
\end{proof}

\section{Proof of Theorem \ref{main_result}}\label{sec_3}
By Proposition \ref{prop_s3} we can focus on $p=\infty$, and extend the conclusions to other values of $p$.

\noindent {\bf Part (a)}. We assume that $\sigma(\vv)=\infty$ and need to show that $\Lambda$ is a ST
for $V^\infty_a(\R^2)$ if and only if $D^{-}(\Gamma)>0$. The necessity of the density condition follows from Lemma \ref{l2}. For the sufficiency we invoke the characterization of sampling trajectories by weak limits (Proposition \ref{prop_s3}). Since each such weak limit satisfies the same conditions that $\Lambda$ does (here it is important that $\Gamma$ be separated), it suffices to show that $\Lambda$ is a US for $V^\infty_a(\R^2)$ whenever $D^{-}(\Gamma)>0$. This follows from Theorem \ref{t_uniq}, part (i).

\noindent {\bf Part (b)}. Assume  that $\sigma(\vv)<\infty$ and suppose first that $D^-(\Gamma)> 1/\sigma(\vv)$. As before, by Proposition \ref{prop_s3}, after taking weak limits, it is enough to show that $\Lambda$ is a US for $V^\infty_a(\R^2)$. This is proved in Theorem \ref{t_uniq}, part (ii).

Second, suppose that $D^-(\Gamma)<1/\sigma(\vv)$. Since
$D^-(\Gamma/\sigma(\vv))=\sigma(\vv) D^-(\Gamma)<1$, by Theorem \ref{th_i3} and Proposition \ref{prop_s1}, $\Gamma/\sigma(\vv)$ is not a SS for $V^\infty_{a\sigma(\vv)^2}(\mathbb{R})$. Hence, there exists a sequence of functions
\begin{align}\label{eq_p2}
g^k (\zeta) = \sum_{n \in \Z} c^k_n e^{-a\sigma(\vv)^2(\zeta-n)^2}  \in V^\infty_{\sigma(\vv)^2a}(\mathbb{R})
\end{align}
with
\begin{align*}
1= \|c^k\|_\infty \asymp \|g^k\|_\infty
\end{align*}
such that
\begin{align*}
\|g^k\,|_{\Gamma/\sigma(\vv)}\|_\infty \to 0, \mbox{ as } k \to \infty.
\end{align*}
Since $\sigma(\vv) < \infty$, there exist $p,q \in \Z$ such that $\sigma(\vv) = \sqrt{p^2+q^2}$, $\gcd(p,q)=1$, and $\frac{pz+qw}{\sigma} \in \Gamma$ iff $(z,w) \in \Lambda.$
Let us set
\begin{align*}
f^k(z,w):=\exp\left\{a\left(\frac{pz+qw}{\sigma(\vv)}\right)^2-az^2-aw^2\right\}g^k\left(\frac{pz+qw}{\sigma(\vv)^2}\right),
\end{align*}
and observe the following.
\begin{claim}\label{claim_f_v}
    Each function $f^k$ admits the representation
    \begin{equation}\label{eq_p1}
        f^k(z,w) = \sum_{n\in\mathbb Z}c^k_n e^{-a(z-pn)^2-a(w-qn)^2},
    \end{equation}
with uniform convergence on compact sets.
\end{claim}
\begin{proof}
    We write
     $$
    f^k(z,w) = \sum_{n \in \Z} c^k_n \exp\left\{a\left(\frac{pz+qw}{\sigma(\vv)}\right)^2-az^2-aw^2-a\sigma(\vv)^2 \left(\frac{pz+qw}{\sigma(\vv)^2} -n \right)^2\right\} 
    $$
    and expand the exponent as
    \begin{align*}
    &\left(\frac{pz+qw}{\sigma(\vv)}\right)^2 - z^2 - w^2 - \sigma(\vv)^2 \left(\frac{pz+qw}{\sigma(\vv)^2} - n\right)^2 
    \\
    &\qquad= -z^2-w^2+2(pz+qw)n - (p^2+q^2)n^2 =       
    -(z-pn)^2-(w-qn)^2,
    \end{align*}
    which formally gives \eqref{eq_p1}. If $|z|, |w| \leq R$, then $|pz+qw| \lesssim R$, so the locally uniform convergence of the series \eqref{eq_p1} follows from that of \eqref{eq_p2}.
\end{proof}
Thus, since $c^k \in \ell^\infty(\mathbb{Z})$, we have that $f^k \in V^\infty_a(\mathbb{R}^2)$. In addition, by Cauchy-Schwarz, 
\begin{align}\label{eq_cs}
\left(\frac{pz+qw}{\sigma(\vv)}\right)^2 \leq z^2+w^2, \qquad z,w \in \R.
\end{align}
Therefore, $\|f^k |_\Lambda \|_\infty \leq 
\|g^k\,|_{\Gamma/\sigma(\vv)} \|_\infty \to 0$, while $\|f^k\|_\infty \asymp \|c^k\|_\infty =1$. Hence, $\Lambda$ is not a ST for $V^\infty_a(\R^2)$. The conclusion extends to all $V^p_a(\R^2)$, $1 \leq p \leq \infty,$ by Proposition \ref{prop_s3}.

\noindent {\bf Part (c)}. Assume that $D^+(\Gamma)<1/\sigma(\vv)$. Since $V^1_a(\R^2) \subseteq V^p_a(\R^2)$, it suffices to show that the set $\La$ in \eqref{la} is not a US for $V^1_a(\R^2)$. 

Observe that the set $\Gamma/\sigma(\vv)$ satisfies $D^+(\Gamma/\sigma(\vv))=\sigma(\vv) D^+(\Gamma)<1.$ Inspecting \eqref{dplus}, we see that for every $\rho>\sigma(\vv)D^+(\Gamma)$ there is an adequate constant $K$ such that \eqref{gp} holds for the set $\Gamma/\sigma(\vv)$. In particular, we may choose such a $\rho$ satisfying $\rho<1.$ Hence, by Proposition~\ref{prop_1d}, there exists a non-trivial function $g$ of the form
$$g(\zeta)=\sum_{n\in\mathbb Z}c_ne^{-a\sigma(\vv)^2(\zeta-n)^2},\quad \{c_n\}\in l^1(\Z),$$which vanishes on $\Gamma/\sigma(\vv).$
As before, we set
\begin{equation}\label{proj_f}
f(z,w):=\exp\left\{a\left(\frac{pz+qw}{\sigma(\vv)}\right)^2-az^2-aw^2\right\}g\left(\frac{pz+qw}{\sigma(\vv)^2}\right)
\end{equation}
and observe that this function vanishes for  $(p/\sigma^2(\vv),q/\sigma^2(\vv))\cdot(z,w)\in\Gamma/\sigma(\vv)$, which means that it vanishes on $ \Lambda$. In addition, by Claim~\ref{claim_f_v} ---
with $f$ in lieu of $f^k$ and $c$ in lieu of $c^k$ ---
we conclude that $f\in V^1_a(\R^2).$  This shows that $\Lambda$ is not a US for $V^1_a(\R^2)$, and concludes the proof of the theorem.
\qed

\section{Proof of Theorem \ref{th_samp}}\label{sec_4}
If $p=0$ or $q=0$, then $\sigma = 1$, and the conclusion follows from \eqref{eq_i6}.

\noindent {\bf Part (a)}. We first assume that $D^{-}(\Gamma_1) > 1,  D^{-}(\Gamma_2) >1$.
By Proposition \ref{prop_s1} it is enough to show that each weak limit of translates of $\Lambda$ is a uniqueness set for $V^\infty_a(\R^2)$. Since each such set satisfies similar conditions as $\Lambda$, we focus on showing that $\Lambda$ is a uniqueness set for $V^\infty_a(\R^2)$.

The following argument refines the one of the proof of Theorem \ref{t_uniq}, part (ii). Let
\begin{align*}
    f(x,y) = \sum_{(n,m) \in \mathbb{Z}^2} c_{n,m} e^{-a(x-n)^2-a(y-m)^2},
\end{align*}
with $c \in \ell^\infty(\Z^2)$ and assume that $f \equiv 0$ on $\Lambda$. Let $(u,v) \in \tfrac{1}{\sigma}\Gamma_1 \times \sigma \Gamma_2$.
Then
\begin{equation}\label{eq_func_rot1}
0 = f(\tfrac{p}{\sigma} u - \tfrac{q}{\sigma} v, \tfrac{q}{\sigma} u + \tfrac{p}{\sigma} v)
= \sum_{(n,m) \in \mathbb{Z}^2} c_{n,m} 
\exp\Big\{-a( \tfrac{p}{\sigma} u - \tfrac{q}{\sigma} v-n)^2-a(\tfrac{q}{\sigma} u + \tfrac{p}{\sigma} v-m)^2\Big\}.
\end{equation}
Let us denote
$$
R := \begin{bmatrix} p/\sigma & - q/\sigma\\
q/\sigma & p/\sigma
\end{bmatrix} .
$$
Using $R^{T}R=I$, where $I$ is identity matrix, 
\begin{align*}
\left|
R\begin{bmatrix} 
u \\ v
\end{bmatrix} - \begin{bmatrix} 
n \\ m
\end{bmatrix}\right| = \left|R^T \left( R\begin{bmatrix} 
u \\ v
\end{bmatrix} - \begin{bmatrix} 
n \\ m
\end{bmatrix} \right) \right| =
\left| \begin{bmatrix} 
u \\ v
\end{bmatrix} - R^T \begin{bmatrix} 
n \\ m
\end{bmatrix}
\right|.
\end{align*}
We can rewrite this relation as
$$
( \tfrac{p}{\sigma} u - \tfrac{q}{\sigma} v-n)^2-(\tfrac{q}{\sigma} u + \tfrac{p}{\sigma} v-m)^2 = (u-( \tfrac{p}{\sigma}n + \tfrac{q}{\sigma}m))^2 +
(v-(\tfrac{q}{\sigma}n-\tfrac{p}{\sigma}m))^2,
$$
which together with~\eqref{eq_func_rot1} implies
$$0=\sum_{(n,m) \in \mathbb{Z}^2} c_{n,m} 
\exp\Big\{-a
\big[(u-( \tfrac{p}{\sigma}n + \tfrac{q}{\sigma}m))^2 +
(v-(\tfrac{q}{\sigma}n-\tfrac{p}{\sigma}m))^2
\big]\Big\}.
$$

Since $\gcd(p,q)=1$, for every $k\in\mathbb Z$, we can choose $n_k,m_k \in \mathbb{Z}$ such that
$pn_k+qm_k=k$. Then $pn+qm=k$ if and only if $n=n_k+ql,m=m_k-pl$, for a uniquely determined $l\in \mathbb Z$. Thus,
\begin{align*}
0&=\sum_{k \in \mathbb{Z}} \sum_{l \in \mathbb{Z}} 
c_{n_k + ql, m_k - pl}
\exp\Big\{-a
\big[(u-\tfrac{k}{\sigma})^2 +
[v-(\tfrac{q}{\sigma}(n_k+ql)-\tfrac{p}{\sigma}(m_k-pl))]^2
\big]\Big\}
\\
&=\sum_{k \in \mathbb{Z}} d_k(v) \exp\big\{-a(u-\tfrac{k}{\sigma})^2\big\},
\end{align*}
where
\begin{align*}
d_k(v) = \sum_{l \in \mathbb{Z}} 
c_{n_k + ql, m_k - pl}
\exp\Big\{-a
\big[v-\tfrac{q}{\sigma}n_k+\tfrac{p}{\sigma}m_k-\sigma l\big]^2\Big\}.
\end{align*}

Let us note that, for each $v \in \sigma \Gamma_2$, $\{d_k(v)\}_{k \in \Z} \in \ell^\infty(\Z)$. Indeed, clearly
$$
|d_k(v)| \le
\|c\|_{\infty}\sup_{t\in\R} \sum_{l \in \mathbb{Z}} 
\exp\Big\{-a(\sigma l-t)^2\Big\} < \infty.
$$

Thus, for each $v \in \sigma \Gamma_2$, the function
\begin{align*}
f_v(t) := \sum_{k \in \mathbb{Z}} d_k(v) \exp\Big\{-\frac{a}{\sigma}(t-k)^2\Big\}
\end{align*}
belongs to the space $V^\infty_{a/\sigma}(\mathbb{R})$ and vanishes on the set
$\Gamma_1$. By Theorem \ref{th_i3} and Proposition \ref{prop_s1}, we conclude that $f_v \equiv 0$ and therefore
\begin{align*}
    d_k(v)=0, \qquad k \in \Z,\, v \in \sigma \Gamma_2.
\end{align*}
As a consequence, for each $k \in \Z$, the function
\begin{align*}
f^k(t) := \sum_{l \in \mathbb{Z}} 
c_{n_k + ql, m_k - pl}
\exp\big\{-a \sigma^2
(t-l)^2\big\}
\end{align*}
vanishes on the set $\Gamma_2-\tfrac{q}{\sigma^2}n_k+\tfrac{p}{\sigma^2}m_k$, which is separated and has lower Beurling density $>1$, while $f^k \in V^\infty_{{\sigma}^2a}(\mathbb{R})$ since $c \in \ell^\infty(\Z)$. By Theorem \ref{th_i3} and Proposition \ref{prop_s1} we conclude that
\begin{align*}
 c_{n_k + ql, m_k - pl} = 0, \qquad k,l \in \Z, 
\end{align*}
and therefore $f \equiv 0$. 

To finish the proof of part (a), it remains to consider the case $D^{-}(\Gamma_1) > \frac{1}{\sigma^2}, D^{-}(\Gamma_2) > \sigma^2 $. To this end, observe that for any $\gamma_1 \in \Gamma_1$ and $\gamma_2 \in \Gamma_2$ we have
$$
\begin{bmatrix} p/\sigma & - q/\sigma\\
q/\sigma & p/\sigma
\end{bmatrix}
\begin{bmatrix} \gamma_1/\sigma \\
\sigma\gamma_2
\end{bmatrix} = 
\begin{bmatrix} p/\sigma & - q/\sigma\\
q/\sigma & p/\sigma
\end{bmatrix}
\begin{bmatrix} 0 & 1\\
-1 & 0
\end{bmatrix}
\begin{bmatrix}  -\sigma\gamma_2 \\
\gamma_1/\sigma
\end{bmatrix} =
\begin{bmatrix} q/\sigma & p/\sigma\\
-p/\sigma & q/\sigma
\end{bmatrix}
\begin{bmatrix}  -\sigma\gamma_2 \\
\gamma_1/\sigma
\end{bmatrix}.$$
Therefore, we see that $\Lambda$ admits the representation
\begin{align}\label{eq_lattice}
\Lambda = \begin{bmatrix} q/\sigma & p/\sigma\\
-p/\sigma & q/\sigma
\end{bmatrix} \cdot \tfrac{1}{\sigma} (-\sigma^2 \Gamma_2) \times \sigma \left(\frac{\Gamma_1}{\sigma^2}\right).
\end{align}
Since $D^{-}\left(\frac{\Gamma_1}{\sigma^2}\right) > 1$ and $D^{-}\left({\sigma^2}\Gamma_2\right) > 1$, 
the conclusion follows from the already considered case.

\noindent {\bf Part (b)}. Assume that $\Lambda$ is a sampling set for $V^2_a(\mathbb{R}^2)$. Then, by Proposition \ref{prop_s1}, $\Lambda$ is a also a sampling set for $V^\infty_a(\mathbb{R}^2)$. Consider the set $\Gamma= \sigma \Gamma_2$ and the unit vector $\vv=(-q/\sigma,p/\sigma)$ and note that $\Lambda$ is included in the following family of parallel lines:
\begin{align}\label{eq_b}
\{(x,y) \in \mathbb{R}^2: (x,y)\cdot (-q/\sigma,p/\sigma) \in \sigma\Gamma_2\} = \begin{bmatrix} p/\sigma & - q/\sigma\\
q/\sigma & p/\sigma
\end{bmatrix} \cdot \mathbb{R} \times \sigma \Gamma_2.
\end{align}
Since $\Lambda$ is a sampling set for $V^\infty_a(\mathbb{R}^2)$, so is the collection of lines \eqref{eq_b}, and Theorem \ref{main_result} implies that
\begin{align*}
D^{-}(\Gamma_2) = \sigma D^{-}(\Gamma) \geq 1.
\end{align*}
Second, the alternative representation of $\Lambda$ in \eqref{eq_lattice} shows that also
\begin{align*}
\sigma^2 D^-(\Gamma_1) = D^{-}\left(\frac{\Gamma_1}{\sigma^2}\right) \geq 1.
\end{align*}
Finally, Landau's necessary density conditions --- see Section \ref{sec_p} --- imply that \[D^-(\Gamma_1) D^-(\Gamma_2) = D^-(\Gamma_1 \times \Gamma_2) \geq 1.\]

\noindent {\bf Part (c)}.
Now we give some examples for the cases of critical densities. Assume that $(p,q) \in \Z^2$ and $\gcd(p,q) = 1$. It suffices to show that the sets $\Lambda'$ and $\Lambda''$ given by
$$
\Lambda' = \begin{bmatrix} p/\sigma & - q/\sigma\\
q/\sigma & p/\sigma
\end{bmatrix} \cdot \tfrac{1}{\sigma} (\sigma^2 (\Z+1/2)) \times \sigma \Gamma_2'. \quad \text{and} \quad \Lambda'' = \begin{bmatrix} p/\sigma & - q/\sigma\\
q/\sigma & p/\sigma
\end{bmatrix} \cdot \tfrac{1}{\sigma} \Gamma_1' \times \sigma (\Z+1/2)
$$
are not US for $V^{\infty}_a(\R^2)$ for any choice of $\Gamma_1',\Gamma_2' \subset \R$. Fix two such sets $\Gamma_1',\Gamma_2'$, and let us construct functions $f_1, f_2\in V^{\infty}_a(\R^2)$ such that $f_1|_{\Lambda'} = 0$ and $f_2|_{\Lambda''} = 0$.

We first note that the function
$$
g(x):= \sum\limits_{n \in \Z} (-1)^n e^{-a\sigma^2(x-n)^2}
$$
belongs to $V^\infty_{a\sigma^2}(\R)$ and vanishes on $\Z + \frac{1}{2}$. (This follows from an elementary computation, or from basic properties of Jacobi's theta function \cite[Chapter 16]{MR1225604}.)

Second, if $\Lambda$ has the form \eqref{eq_lambdaset}, then each point $(\lambda_1, \lambda_2) \in \Lambda$ can be written as
$$(\lambda_1, \lambda_2) = (p\gamma_1/\sigma^2 - q \gamma_2, q\gamma_1/\sigma^2 + p \gamma_2),$$ where $\gamma_j \in \Gamma_j, j=1,2,$ are such that
$$
\frac{p\lambda_2 - q \lambda_1}{\sigma^2} = \left(-\frac{qp}{\sigma^4} + \frac{pq}{\sigma^4}\right)\gamma_1 + \frac{q^2+p^2}{\sigma^2} \gamma_2 = \gamma_2.
$$
$$
\frac{p \lambda_1+q\lambda_2}{\sigma^2} = \left(\frac{p^2}{\sigma^4} + \frac{q^2}{\sigma^4}\right)\gamma_1 + \frac{-pq+pq}{\sigma^2} \gamma_2 = \frac{\gamma_1}{\sigma^2}.
$$
Therefore, the function
$$
f_1(z,w):=g\left(\frac{pz + qw}{\sigma^2}\right) \exp\left\{a\left(\frac{pz+qw}{\sigma(\vv)}\right)^2-az^2-aw^2\right\}
$$
vanishes on $\Lambda'$ while the function
$$
f_2(z,w):=g\left(\frac{pw - qz}{\sigma^2}\right) \exp\left\{a\left(\frac{pw-qz}{\sigma(\vv)}\right)^2-az^2-aw^2\right\}
$$
vanishes on $\Lambda''$. Similarly to Claim \ref{claim_f_v}, one can check that $f_1$ and $f_2$ belong to $V^{\infty}_a(\R^2)$.
\qed

\section{Gabor Frames for certain rational two-dimensional lattices}\label{sec_5}

To connect the sampling theorems for shift-invariant spaces with Gabor frames, we need the following result which can be traced back to Janssen \cite{jan95}. For concreteness, we cite a simplified version of \cite[Theorem~2.3]{grs}.\footnote{The result in \cite{grs} is formulated in dimension $d=1$, but the arguments extend verbatim to $d=2$.}

\begin{proposition}\label{prop_gabor_sampling}
    Let $\Lambda \subset \R^2$ be separated, $a>0$, and consider the Gaussian function $g_a(x,y) = e^{-a(x^2+y^2)}, (x,y)\in\R^2$.
        Then the following statements are equivalent:
    \begin{enumerate}
        \item[(i)] The Gabor system $\mathcal{G}(g_a,(-\Lambda)\times \Z^2)$ is a frame of $L^2(\R^2).$
        \item[(ii)] The set $\Lambda + (u,v)$ is a sampling set of $V^2_a(\R^2)$ for every $(u,v) \in [0,1)^2.$ 
    \end{enumerate}      
\end{proposition}

We can now prove the main result on Gabor frames.
\begin{proof}[Proof of Theorem~\ref{Gabor_thm}]

\noindent {\bf Part (a)}.
    Observe that for any $(u,v) \in [0,1)^2$ the set $-\Lambda + (u,v)$ admits the representation~\eqref{eq_lambdaset} with 
    $$\Gamma_1'= -\Gamma_1 + pu +  qv, \quad \Gamma_2' = -\Gamma_2 -qu/\sigma^2 + pv/\sigma^2.$$
    
    Since $D^-(\Gamma_1') = D^-(\Gamma_1)$ and $D^-(\Gamma_2') = D^-(\Gamma_2)$, by Theorem~\ref{th_samp}, Part (a), we deduce that all translates $-\Lambda + x$, $x \in [0,1)^2,$ are sampling sets for $V^2_a(\R^2).$ Together with Proposition~\ref{prop_gabor_sampling}, this shows that $\mathcal{G}(g_a,\Lambda\times \Z^2)$ is a frame of $L^2(\R^2)$.
    
\noindent {\bf Part (b)}. Suppose that 
$\mathcal{G}(g_a,\Lambda\times \Z^2)$ is a frame of $L^2(\R^2)$. By Proposition~\ref{prop_gabor_sampling} and Part (b) of Theorem~\ref{th_samp}, we conclude that the density conditions
\eqref{eq_not_dens} hold. To prove that in fact the strict density conditions \eqref{eq_not_dens_2} hold we invoke the stability of Gabor frames under small deformations of the time-frequency lattice, which is applicable because the Gaussian function $g_a$ is smooth and fast decaying \cite{MR3192621,bel,fk,MR3336091}. More precisely, consider the matrices
\begin{align*}
A_r = \begin{bmatrix}
    r & 0 & 0 & 0\\
    0 & r & 0 & 0\\
    0 & 0 & 1 & 0\\
    0 & 0 & 0 & 1
\end{bmatrix}, \qquad r >0.
\end{align*}
By \cite[Theorem 1.1]{fk} there exists $r>1$ such that
\begin{align*}
\mathcal{G}(g_a,A_r(\Lambda\times \Z^2))
=\mathcal{G}(g_a,(r\Lambda)\times \Z^2)
=\mathcal{G}(g_a,\Lambda_r\times \Z^2)
\end{align*}
is a frame, where $\Lambda_r$ is associated with $r\Gamma_1$ and $r\Gamma_2$ by \eqref{eq_lambdaset}. Invoking again Proposition~\ref{prop_gabor_sampling} and Theorem~\ref{th_samp} we obtain
\begin{equation*}
r^{-1}D^{-}(\Gamma_1) \ge \frac{1}{\sigma^2}, 
\quad r^{-1}D^{-}(\Gamma_2) \ge 1, \quad\mbox{and} \quad r^{-1}D^{-}(\Gamma_1) D^{-}(\Gamma_2) \ge 1,
\end{equation*}
which implies \eqref{eq_not_dens_2}.
\end{proof}

Finally, we apply the previous theorem to lattices.
\begin{proof}[Proof of Corollary \ref{coro_gabor}]
The case $a=b=1$ follows directly from Theorem \ref{Gabor_thm}. Consider now $\alpha>0$ and general $a,b>0$ and suppose that either $c<1$ and $d<1$, or $c<\sigma^2$ and $d<\sigma^{-2}$. By the previous case,
the Gabor system $\mathcal{G}(g_{\alpha/a^2, \alpha/a^2}, \Delta_{1,1,c,d})$ is a frame of $L^2(\R^2)$. A direct calculation shows that the isomorphism $L^2(\R^2) \ni f(x,y) \mapsto f(ax,by) \in L^2(\R^2)$ maps $\mathcal{G}(g_{\alpha/a^2, \alpha/a^2}, \Delta_{1,1,c,d})$ into
$\mathcal{G}(g_{\alpha, \alpha b^2/a^2}, \Delta_{a,b,c,d})$, and thus the latter Gabor system is also a frame of $L^2(\R^2)$.
\end{proof}

\section{Extension to other generators}\label{sec_other}
We consider a continuous function $\varphi: \mathbb{R}^2 \to \mathbb{C}$ with
\begin{align}\label{am}
\sum_{(n,m)\in\mathbb{Z}^2} \sup_{(x,y)\in[0,1]^2} |\varphi(x+n,y+m)| < \infty
\end{align}
and look into extending our results to the \emph{shift-invariant space generated by $\varphi$}:
\begin{align}\label{eq_i2g}
V^2(\varphi):=\Big\{f(x,y)=\sum_{(n,m)\in\mathbb Z^2}c_{n,m}\varphi(x-n,y-m):c\in \ell^2(\mathbb Z^2)\Big\}.
\end{align}
Condition \eqref{am} implies that the series in \eqref{eq_i2g} converge in $L^2$. The \emph{Zak transform} of $\varphi$,
\begin{align}
Z\varphi(t,\omega) = \sum_{(n,m) \in \mathbb{Z}^2} \varphi(t_1+m,t_2+n) e^{2\pi i (n \omega_1+m\omega_2)}, \quad (t,\omega)=(t_1,t_2,\omega_1,\omega_2) \in \mathbb{R}^2\times\mathbb{R}^2,
\end{align}
is useful to investigate $V^2(\varphi)$. We mention a few facts that are easily proved; see 
\cite{jan95} or \cite[Theorem 2.3]{grs} for closely related statements.
\begin{lemma}\label{lem_new}
Let $\varphi: \mathbb{R}^2 \to \mathbb{C}$ be continuous and satisfy \eqref{am} and 
\[\inf_{(t,\omega) \in \mathbb{R}^2\times\mathbb{R}^2} |Z\varphi(t,\omega)| >0.\] Then the following hold.
\begin{itemize}
\item[(i)] The system of translates
$\{\varphi(x-n,y-m): (n,m)\in\mathbb{Z}^2\}$ is a Riesz basis of $V^2(\varphi)$, that is, every $f \in V^2(\varphi)$ admits a unique expansion as in \eqref{eq_i2g}, and, moreover, $\|f\|_2 \asymp \|c\|_2$. In addition, every $f \in V^2(\varphi)$ is continuous.

\item[(ii)] Let $\Lambda \subseteq \mathbb{R}^2$ be a finite union of separated sets. Then $\Lambda$ is a sampling set for $V^2(\varphi)$ if and only if $\{Z\varphi(\lambda,\omega): \lambda \in \Lambda\}$ is a frame of $L^2([0,1]^2,d\omega)$, that is, if and only if there exist $A,B>0$ such that
\begin{align}
A \|h\|^2_2 \leq \sum_{\lambda \in \Lambda} \Big| 
\int_{[0,1]^2} Z\varphi(\lambda,\omega) \overline{h(\omega)}\,d\omega
\Big|^2 \leq B \|h\|^2_2, \qquad {\text{for every  }\,\,} h \in L^2([0,1]^2).
\end{align}
\end{itemize}
\end{lemma}

As an application, we obtain the following extension of Theorem \ref{th_samp}.

\begin{corollary}\label{coro_new}
Let $\varphi: \mathbb{R}^2 \to \mathbb{C}$ be continuous and satisfy \eqref{am}, and consider also the Gaussian function $g_a(x,y)=e^{-a(x^2+y^2)}$ with $a>0$.

Suppose that there exist measurable functions $p,q: \mathbb{R} \to \mathbb{C}$
with
\begin{align}\label{eq_aaa1}
0 < \inf_{t \in \mathbb{R}^2} |p(t)| \leq 
\sup_{t \in \mathbb{R}^2} |p(t)| < \infty,
\\\label{eq_aaa2}
0 < \inf_{\omega \in \mathbb{R}^2} |q(\omega)| \leq 
\sup_{\omega \in \mathbb{R}^2} |q(\omega)| < \infty,
\end{align}
and 
\begin{align}\label{eq_aaa3}
Z\varphi(t,\omega) = p(t) q(\omega) Z{g_a} (t,\omega), \qquad (t,\omega) \in \mathbb{R}^2\times\mathbb{R}^2.
\end{align}
Then Theorem \ref{th_samp} holds with $V^2(\varphi)$ in lieu of $V^2_a(\mathbb{R}^2)$.
\end{corollary}
\begin{proof}
Using Lemma \ref{lem_new} and \eqref{eq_aaa1}, \eqref{eq_aaa2}, \eqref{eq_aaa3},
we see that a separated set $\Lambda \subseteq \mathbb{R}^2$ is a sampling set for $V^2(\varphi)$ if and only if it is a sampling set for $V^2_a(\mathbb{R}^2)$. Indeed, if
$\Lambda$ is a sampling set for 
$V_a^2(\mathbb{R}^2)$, then
\begin{align*}
\sum_{\lambda \in \Lambda} \Big| 
\int_{[0,1]^2} Z\varphi(\lambda,\omega) \overline{h(\omega)}\,d\omega
\Big|^2  &= \sum_{\lambda \in \Lambda} \Big| 
\int_{[0,1]^2} Zg_{a}(\lambda,\omega)p(\lambda) q(w) \overline{h(\omega)}\,d\omega
\Big|^2 
\\
&\asymp
\sum_{\lambda \in \Lambda} \Big| 
\int_{[0,1]^2} Zg_{a}(\lambda,\omega) q(w) \overline{h(\omega)}\,d\omega
\Big|^2 
\asymp \|\overline{q} h\|^2_2 \asymp \|h\|^2_2,
\end{align*}
and $\Lambda$ is also a sampling set for $V^2(\varphi)$. The converse implication is proved similarly.
\end{proof}
Similarly, though we do not provide the details here, Theorem \ref{Gabor_thm} and Corollary \ref{coro_gabor} can be extended to functions $\varphi$ as in Corollary \ref{coro_new} in lieu of the Gaussian function $g_a$, as long as $\varphi$ belongs to the so-called modulation space $M^1(\mathbb{R}^2)$ of functions with integrable short-time Fourier transforms (and $Z \varphi$ satisfies~\eqref{eq_aaa3}, \eqref{eq_aaa1}, and \eqref{eq_aaa2}). 

\begin{example}[Hyperbolic secant] Fix $a>0$ and let $\psi(t) := \frac{1}{e^{at}+e^{-at}}$ and $h(t) := e^{-at^2}$,
$t \in \mathbb{R}$. 
Janssen and Strohmer \cite{MR1884237} proved that the Zak transforms of $\psi$ and $h$, 
\begin{align*}
Z\psi(t,\omega) = \sum_{k\in\mathbb{Z}} \psi(t+k) e^{2\pi i k \omega}, \qquad 
Zh(t,\omega) = \sum_{k\in\mathbb{Z}} h(t+k) e^{2\pi i k \omega}, \qquad t,\omega\in\mathbb{R},
\end{align*}
are related by $Z\psi(t,\omega)=p(t) q(\omega) Z h(\omega)$ with $|p|,|q|$ bounded from above and below.

Let $\varphi(x,y)=\psi(x) \psi(y) = \frac{1}{(e^{ax}+e^{-ax})(e^{ay}+e^{-ay})}$. Then
\begin{align*}
Z\varphi(t_1,t_2,\omega_1,\omega_2) = Z\psi(t_1,\omega_1) Z\psi(t_2,\omega_2) = p(t_1) p(t_2) q(\omega_1) q(\omega_2) Zg_a(t_1,t_2,\omega_1,\omega_2)
\end{align*}
and the hypotheses of Corollary \ref{coro_new} are satisfied. Hence Theorem \ref{th_samp} holds with $V^2(\varphi)$ in lieu of $V^2_a(\mathbb{R}^2)$.
\end{example}

\section{Postponed proofs}\label{sec_pos}
We now prove Proposition \ref{prop_s3}
by reducing it to Proposition \ref{prop_s1} by the next discretization lemma. Similar arguments can be found in \cite{MR3935271,Jaming,JOC}.

\begin{lemma}\label{lemma_s2}
Let $\Gamma \subset \R$ be separated, $a>0$, and $1 \leq p \leq \infty$. Then the collection of lines \eqref{la} is a ST for $V^p_a(\R^2)$ if and only if there exists a separated set $\Lambda^0 \subset \Lambda$ that is a SS for $V^p_a(\R^2)$.
\end{lemma}
\begin{proof}
We focus on the case $p<\infty$; the arguments for $p=\infty$ are simpler and omitted. Consider the slanted cubes
$$Q_{k,\delta} := \begin{bmatrix} p/\sigma &  q/\sigma\\
-q/\sigma & p/\sigma
\end{bmatrix} \cdot \Big(
[-\delta/2, \delta/2)^2+ \delta k\Big), \qquad k \in \Z^2, $$
where $0<\delta < \delta(\Gamma)/3$.

\noindent {\bf Step 1}. Each function $f \in V^p_a(\R^2)$ admits a unique expansion
\begin{equation}\label{rb}
f=\sum_{k\in\Z^2} c_k g_a(\cdot-k), \qquad\mbox{with } 
\|c\|_p \asymp \|f\|_p.
\end{equation}
Since $g_a$ and $\nabla g_a$ are fast decaying, a straightforward argument then shows that 
\begin{align}\label{eq_a6}
\sum_{k \in \Z^2} \big( \sup_{Q_{k,1}}|f|^p + \sup_{Q_{k,1}}|\nabla f|^p \big) \lesssim \|f\|_p^p, \qquad f \in V^p_a(\R^2),
\end{align}
see, e.g., \cite{ag01}. As a consequence, for every $\eta$-separated set $\Lambda^0 \subset \R^2$, $\eta>0$, the upper sampling inequality
\begin{align}\label{eq_s5}
\sum_{\lambda \in \Lambda^0} |f(\lambda)|^p \leq C
(1+1/\eta)^p \|f\|^p_p, \qquad f \in V^p_a(\R^2),
\end{align}
holds with universal constant $C>0$.

\noindent {\bf Step 2}. 
Set $I_{k,\delta} := \{k \in \Z: Q_{k,\delta} \cap \Lambda \not= \emptyset\}$. For each $k \in I_{k,\delta}$ select a point $\lambda_{k,\delta} \in Q_{k,\delta} \cap \Lambda$,
in such a way that the resulting set $\Lambda^0_{\delta} := \{\lambda_{k,\delta}: k \in I_{k,\delta}\}$ is separated.

We claim that for small enough $\delta$, $\Lambda^0_{\delta}$ is a SS for $V^p_a(\R^2)$.
Note first that
\begin{align*}
|f(\lambda) - f(\lambda_{k,\delta})| \lesssim \sup_{Q_{k,\delta}} |\nabla f| \delta, \qquad \lambda \in Q_{k,\delta} \cap \Lambda,
\end{align*}
while the arc-length of $Q_{k,\delta} \cap \Lambda$ is 
$\lesssim \delta$.
On the other hand,
\begin{align}
\sup_{j \in \Z} \# \{k \in I_{k,\delta}: Q_{k,\delta} \cap Q_{j,1} \not=\emptyset \} \leq C_{\Lambda} \delta^{-1}.
\end{align}
Combining this with \eqref{eq_a6}, we estimate for $f \in V^p_a(\R^2)$,
\begin{align}\label{eq_a4}
\begin{aligned}
\int_\Lambda |f|^p \,ds &\lesssim 
\delta \sum_{k \in I_{k,\delta}} |f(\lambda_{k, \delta})|^p + C \delta^{1+p} 
\sum_{k \in I_{k,\delta}} \sup_{Q_{k,\delta}} |\nabla f|^p
\\
&\leq \delta \sum_{k \in I_{k,\delta}} |f(\lambda_{k, \delta})|^p + C'_\Lambda \delta^p \sum_{j \in I_{k,1}} \sup_{Q_{j,1}} |\nabla f|^p
\\
&\leq \delta \sum_{k \in I_{k,\delta}} |f(\lambda_{k, \delta})|^p + C''_\Lambda \delta^p \|f\|^p_p.
\end{aligned}
\end{align}
\noindent {\bf Step 3}.
Suppose that $\Lambda$ is a ST for $V^p_a(\R^2)$. Hence, there exists $A>0$ such that
\begin{align}\label{eq_a1}
\|f\|^p_p \leq A \int_\Lambda |f|^p \,ds, \qquad f \in V^p_a(\R^2).
\end{align}
Choosing $\delta$ so that $A C''_\Lambda \delta^p < 1/2$ gives a lower sampling inequality for $\Lambda^0_\delta$, while the upper sampling inequality follows from \eqref{eq_s5} because $\Lambda^0_\delta$ is separated.

\noindent {\bf Step 4}. 
Conversely, suppose that there exists a separated set $\Lambda^0 \subset \Lambda$ and $A>0$ such that
\begin{align}
\|f\|^p_p \leq A' \sum_{\lambda \in \Lambda^0} |f(\lambda)|^p, \qquad f \in V^p_a(\R^2),
\end{align}
and set $\eta := \inf\{|\lambda-\mu|_\infty: \lambda,\mu \in \Lambda^0, \lambda \not= \mu\}$
Choosing $\delta < \eta/3$, we see that for each $\lambda \in  \Lambda^0$ there exists a unique $k_\lambda \in \Z^2$ such that $\lambda \in Q_{k_\lambda, \delta}$, while
\begin{align*}
\sup_{j \in \Z^2} \#\{\lambda \in \Lambda^0: Q_{k_\lambda,\delta} \cap Q_{j,1} \not= \emptyset\} \leq C_\eta. 
\end{align*}
Thus, for $f \in V^p_a(\R^2)$,
\begin{align*}
\sum_{\lambda \in \Lambda^0} |f(\lambda)|^p 
&\lesssim \sum_{\lambda \in \Lambda^0} \delta^{-1} \int_{Q_{k_\lambda,\delta} \cap \Lambda} |f|^p \, ds + C \delta^p 
\sum_{\lambda \in \Lambda^0} \sup_{Q_{k_\lambda,\delta}} |\nabla f|^p
\\
&\leq \delta^{-1} \int_\Lambda |f|^p \, ds + C'_\eta  
\delta^p
\sum_{k \in \Z^1} \sup_{Q_{k,\delta}} |\nabla f|^p
\\
&\leq \delta^{-1} \int_\Lambda |f|^p \, ds + C''_\eta  
\delta^p
\|f\|^p_p.
\end{align*}

Hence, if $\delta$ is small enough so that $A'C''_\eta 
\delta^p <1/2$, we conclude that $\Lambda$ satisfies \eqref{eq_a1} for an adequate constant $A>0$. In addition, \eqref{eq_s5}  and \eqref{eq_a4} show that
$\int_\Lambda |f|^p \,ds \lesssim \|f\|_p^p$, and therefore $\Lambda$ is a ST for $V^p_a(\R^2)$.
\end{proof}

We can now prove prove the characterization of sampling trajectories.
\begin{proof}[Proof of Proposition \ref{prop_s3}]
\noindent {\bf (a) $\Leftrightarrow$ (b)} follows directly from Proposition \ref{prop_s1} and Lemma \ref{lemma_s2}.

\noindent {\bf (b) $\Rightarrow$ (c)}. Suppose that $\Lambda$ is a ST for $V^\infty_a(\R^2)$ and let $\Gamma'=\lim_k \Gamma+\vv\cdot(n_k,m_k) \in W_\Sigma(\Gamma)$. By Lemma \ref{lemma_s2} there exists $\Lambda^0 \subset \Lambda$ a separated SS for $V^\infty_a(\R^2)$. By passing to a subsequence, we may assume that $\La^0+(n_k,m_k)$ converges weakly to a set $\La^{1}$. Since
\begin{align*}
\La^0+(n_k,m_k) \subset
\La+(n_k,m_k)=\{(x,y):(x,y)\cdot \vv\in\Gamma+\vv\cdot(n_k,m_k)\},
\end{align*}
it follows that $\La^1 \subset \La'$, as defined in \eqref{la2}. On the other hand, Theorem \ref{th_i3} and Proposition \ref{prop_s1} imply that $\La^{1}$ is a SS for $V^\infty_a(\R^2)$. Hence, the larger set $\La'$ is also a ST for $V^\infty_a(\R^2)$.

\noindent {\bf (c) $\Rightarrow$ (d)} is obvious.

\noindent {\bf (d) $\Rightarrow$ (a)}. Suppose that $\Lambda$ is not a SS for $V^\infty_a(\R^2)$. Taking into account \eqref{rb}, this means that there is a sequence of functions
$$
f_k(x,y):=
\sum_{(n,m) \in \Z^2}c^{(k)}_{n,m} e^{-a(x-n)^2 - a(y-m)^2}, \qquad k \in \N,$$
with $\sup_{(n,m) \in \Z} |c^{(k)}_{n,m}|=1$ and
$\sup_{(\lambda_1,\lambda_2)\in\La}|f_k(\lambda_1,\lambda_2)|<\frac{1}{k}$.
For each $k \in \N$ we select $(n_k,m_k) \in \Z^2$ such that $|c^{(k)}_{n_k,m_k}| > 1/2$ and set $d^{(k)}_{n,m} : =c^{(k)}_{n+n_k,m+m_k}$,
$$
g_k(x,y):=\sum_{(n,m)\in\Z^2}d^{(k)}_{n,m} e^{-a(x-n)^2 - a(y-m)^2}\in V^\infty_a(\R^2),$$
and $\La_k:=\La-(n_k,m_k)=\{(x,y):(x,y)\cdot \vv\in\Gamma-\vv\cdot(n_k,m_k)\}$. Hence, \[\sup_{(\lambda_1,\lambda_2)\in\La_k}|g_k(\lambda_1,\lambda_2)|<\frac{1}{k}.\]
Passing to subsequences, we may assume that (i) $\Gamma+\vv\cdot(-n_k,-m_k)$ converge weakly to some set $\Gamma' \in W_{\Sigma}(\Gamma)$ and that (ii) $d^{(k)}_{n,m} \to d_{n,m}$ as $k \to \infty$. It follows that the functions $g_k$ converge uniformly on compact subsets of $\R^2$ to a function
$$
g(x,y):=\sum_{(n,m)\in\Z^2}d_{n,m} e^{-a(x-n)^2 - a(y-m)^2}\in V^\infty_a(\R^2).$$
This function vanishes on the set $\Lambda'$ defined in \eqref{la2}, and it is not identically zero because $|d_{0,0}| \ge 1/2$. This contradicts (d) and concludes the proof.
\end{proof}

\section{Acknowledgements}
Ilya Zlotnikov is grateful to Markus Faulhuber and Irina Shafkulovska for fruitful discussions. The authors are grateful to the anonymous referee for the helpful remarks.
\bibliographystyle{abbrv}
\bibliography{biblio}

\begin{thebibliography}{10}

\bibitem{MR1225604}
M.~Abramowitz and I.~A. Stegun, editors.
\newblock {\em Handbook of mathematical functions with formulas, graphs, and
  mathematical tables}.
\newblock Dover Publications, Inc., New York, 1992.
\newblock Reprint of the 1972 edition.

\bibitem{ag01}
A.~Aldroubi and K.~Gr\"{o}chenig.
\newblock Nonuniform sampling and reconstruction in shift-invariant spaces.
\newblock {\em SIAM Rev.}, 43(4):585--620, 2001.

\bibitem{MR3192621}
G.~Ascensi, H.~G. Feichtinger, and N.~Kaiblinger.
\newblock Dilation of the {W}eyl symbol and {B}alian-{L}ow theorem.
\newblock {\em Trans. Amer. Math. Soc.}, 366(7):3865--3880, 2014.

\bibitem{MR2224392}
R.~Balan, P.~G. Casazza, C.~Heil, and Z.~Landau.
\newblock Density, overcompleteness, and localization of frames. {I}. {T}heory.
\newblock {\em J. Fourier Anal. Appl.}, 12(2):105--143, 2006.

\bibitem{MR4456797}
A.~Baranov, Y.~Belov, and K.~Gr\"{o}chenig.
\newblock Complete interpolating sequences for the {G}aussian shift-invariant
  space.
\newblock {\em Appl. Comput. Harmon. Anal.}, 61:191--201, 2022.

\bibitem{bel}
J.~Bellissard.
\newblock Lipshitz continuity of gap boundaries for hofstadter-like spectra.
\newblock {\em Comm. Math. Phys.}, 160(3):599–613, 1994.

\bibitem{MR4542702}
Y.~Belov, A.~Kulikov, and Y.~Lyubarskii.
\newblock Gabor frames for rational functions.
\newblock {\em Invent. Math.}, 231(2):431--466, 2023.

\bibitem{bes}
A.~S. Besicovitch.
\newblock {\em Almost periodic functions}.
\newblock Dover Publications, Inc., New York, 1955.

\bibitem{MR0427956}
A.~Beurling.
\newblock Local harmonic analysis with some applications to differential
  operators.
\newblock In {\em Some {R}ecent {A}dvances in the {B}asic {S}ciences, {V}ol. 1
  ({P}roc. {A}nnual {S}ci. {C}onf., {B}elfer {G}rad. {S}chool {S}ci., {Y}eshiva
  {U}niv., {N}ew {Y}ork, 1962--1964)}, pages 109--125. 1966.

\bibitem{MR10576141a}
A.~Beurling.
\newblock {B}alayage of {F}ourier--{S}tieltjes transforms.
\newblock In L.~Carleson, P.~Malliavin, J.~Neuberger, and J.~Wermer, editors,
  {\em The collected works of {A}rne {B}eurling. {V}ol. 2}, Contemporary
  Mathematicians, pages xx+389. Birkh\"{a}user Boston, Inc., Boston, MA, 1989.
\newblock Harmonic analysis.

\bibitem{fk}
H.~G. Feichtinger and N.~Kaiblinger.
\newblock Varying the time-frequency lattice of {G}abor frames.
\newblock {\em Trans. Amer. Math. Soc.}, 356(5):2001--2023, 2004.

\bibitem{MR3742438}
H.~F\"{u}hr, K.~Gr\"{o}chenig, A.~Haimi, A.~Klotz, and J.~L. Romero.
\newblock Density of sampling and interpolation in reproducing kernel {H}ilbert
  spaces.
\newblock {\em J. Lond. Math. Soc. (2)}, 96(3):663--686, 2017.

\bibitem{g_gabor}
K.~Gr\"{o}chenig.
\newblock Multivariate {G}abor frames and sampling of entire functions of
  several variables.
\newblock {\em Appl. Comput. Harmon. Anal.}, 31(2):218--227, 2011.

\bibitem{GL}
K.~Gr\"{o}chenig and Y.~Lyubarskii.
\newblock Sampling of entire functions of several complex variables on a
  lattice and multivariate {G}abor frames.
\newblock {\em Complex Var. Elliptic Equ.}, 65(10):1717--1735, 2020.

\bibitem{MR3336091}
K.~Gr\"{o}chenig, J.~Ortega-Cerd\`a, and J.~L. Romero.
\newblock Deformation of {G}abor systems.
\newblock {\em Adv. Math.}, 277:388--425, 2015.

\bibitem{grs}
K.~Gr\"{o}chenig, J.~L. Romero, and J.~St\"{o}ckler.
\newblock Sampling theorems for shift-invariant spaces, {G}abor frames, and
  totally positive functions.
\newblock {\em Invent. Math.}, 211(3):1119--1148, 2018.

\bibitem{g}
K.~Gr\"{o}chenig, J.~L. Romero, J.~Unnikrishnan, and M.~Vetterli.
\newblock On minimal trajectories for mobile sampling of bandlimited fields.
\newblock {\em Appl. Comput. Harmon. Anal.}, 39(3):487--510, 2015.

\bibitem{MR3053565}
K.~Gr\"{o}chenig and J.~St\"{o}ckler.
\newblock Gabor frames and totally positive functions.
\newblock {\em Duke Math. J.}, 162(6):1003--1031, 2013.

\bibitem{MR3935271}
P.~Jaming and F.~Negreira.
\newblock A {P}lancherel-{P}olya inequality in {B}esov spaces on spaces of
  homogeneous type.
\newblock {\em J. Geom. Anal.}, 29(2):1571--1582, 2019.

\bibitem{Jaming}
P.~Jaming, F.~Negreira, and J.~L. Romero.
\newblock The {N}yquist sampling rate for spiraling curves.
\newblock {\em Appl. Comput. Harmon. Anal.}, 52:198--230, 2021.

\bibitem{jan95}
A.~J. E.~M. Janssen.
\newblock Duality and biorthogonality for {W}eyl-{H}eisenberg frames.
\newblock {\em J. Fourier Anal. Appl.}, 1(4):403--436, 1995.

\bibitem{MR1884237}
A.~J. E.~M. Janssen and T.~Strohmer.
\newblock Hyperbolic secants yield {G}abor frames.
\newblock {\em Appl. Comput. Harmon. Anal.}, 12(2):259--267, 2002.

\bibitem{la67}
H.~J. Landau.
\newblock Necessary density conditions for sampling and interpolation of
  certain entire functions.
\newblock {\em Acta Math.}, 117:37--52, 1967.

\bibitem{levitan}
B.~M. Levitan.
\newblock {\em Po\v{c}ti-periodi\v{c}eskie funkcii}.
\newblock Gosudarstv. Izdat. Tehn.-Teor. Lit., Moscow, 1953.

\bibitem{Luef_gabor}
F.~Luef and X.~Wang.
\newblock Gaussian gabor frames, {S}eshadri constants and generalized
  {B}user{\textendash}{S}arnak invariants.
\newblock {\em Geom. Funct. Anal.}, Apr. 2023.

\bibitem{MR1188007}
Y.~I. Lyubarski\u{\i}.
\newblock Frames in the {B}argmann space of entire functions.
\newblock In {\em Entire and subharmonic functions}, volume~11 of {\em Adv.
  Soviet Math.}, pages 167--180. Amer. Math. Soc., Providence, RI, 1992.

\bibitem{MR2917231}
A.~Olevskii and A.~Ulanovskii.
\newblock On multi-dimensional sampling and interpolation.
\newblock {\em Anal. Math. Phys.}, 2(2):149--170, 2012.

\bibitem{JOC}
J.~Ortega-Cerd\`a.
\newblock Sampling measures.
\newblock {\em Publ. Mat.}, 42(2):559--566, 1998.

\bibitem{PR_gabor}
G.~E. Pfander and P.~Rashkov.
\newblock Remarks on multivariate {G}aussian {G}abor frames.
\newblock {\em Monatsh. Math.}, 172(2):179--187, 2013.

\bibitem{RUZ}
A.~Rashkovskii, A.~Ulanovskii, and I.~Zlotnikov.
\newblock On 2-dimensional mobile sampling.
\newblock {\em Appl. Comput. Harmon. Anal.}, 62:1--23, 2023.

\bibitem{MR1173118}
K.~Seip and R.~Wallst\'{e}n.
\newblock Density theorems for sampling and interpolation in the
  {B}argmann-{F}ock space. {II}.
\newblock {\em J. Reine Angew. Math.}, 429:107--113, 1992.

\bibitem{u1}
J.~Unnikrishnan and M.~Vetterli.
\newblock Sampling and reconstruction of spatial fields using mobile sensors.
\newblock {\em {IEEE} Trans. Signal Process.}, 61(9):2328--2340, May 2013.

\bibitem{u2}
J.~Unnikrishnan and M.~Vetterli.
\newblock Sampling high-dimensional bandlimited fields on low-dimensional
  manifolds.
\newblock {\em {IEEE} Trans. Inf. Theory}, 59(4):2103--2127, Apr. 2013.

\end{thebibliography}

\end{document}